\documentclass{amsart}

\usepackage[utf8]{inputenc}
\usepackage[utf8]{inputenc}
\usepackage{mathrsfs}
\usepackage{amscd,amssymb,amsopn,amsmath,amsthm,graphics,amsfonts,enumerate,verbatim,calc}
\usepackage[all]{xy}
\newtheorem{theorem}{Theorem}[section]
\newtheorem{claim}[theorem]{Claim}

\newtheorem{lemma}[theorem]{Lemma}
\newtheorem{proposition}[theorem]{Proposition}

\newtheorem{corollary}[theorem]{Corollary}\newtheorem{conjecture}[theorem]{Conjecture}

\theoremstyle{definition}
\newtheorem{definition}[theorem]{Definition}

\newtheorem{fact}[theorem]{Fact}

\newtheorem{discussion}[theorem]{Discussion}

\theoremstyle{remark}
\newtheorem{remark}[theorem]{Remark}
\newtheorem{question}[theorem]{Question}
\newtheorem{notation}[theorem]{Notation}

\newcommand{\Ker}{{\rm Ker}}\newcommand{\cok}{{\rm coker}}

\newcommand{\rest}{{\restriction}}
\newcommand{\dom}{{\rm dom}}

\newcommand{\Rang}{{\rm Rang}}
\newcommand{\Ext}{{\rm Ext}}
\newcommand{\Hom}{{\rm Hom}}
\newcommand{\Max}{{\rm Max}}
\newcommand{\supp}{{\rm Supp}}
\newcommand{\suc}{{\rm suc}}
\newcommand{\lub}{{\rm lub}}
\newcommand{\otp}{{\rm otp}}
\newcommand{\id}{{\rm id}}
\newcommand{\wilog}{{\rm without loss of generality}}
\newcommand{\Wilog}{{\rm Without loss of generality}}

\newcommand{\mn}{{\medskip\noindent}}
\newcommand{\sn}{{\smallskip\noindent}}

\newcommand{\cD}{{\mathscr D}}

\newcommand{\cH}{{\mathscr H}}

\newcommand{\cG}{{\mathscr G}}

\newcommand{\bbP}{{\mathbb P}}
\newcommand{\cP}{{\mathscr P}}

\newcommand{\bbQ}{{\mathbb Q}}
\newcommand{\bbZ}{{\mathbb Z}}

\newcommand{\cU}{{\mathscr U}}

\newcommand{\cf}{{\rm cf}}

\newcount\skewfactor
\def\mathunderaccent#1#2 {\let\theaccent#1\skewfactor#2
\mathpalette\putaccentunder}
\def\putaccentunder#1#2{\oalign{$#1#2$\crcr\hidewidth
\vbox to.2ex{\hbox{$#1\skew\skewfactor\theaccent{}$}\vss}\hidewidth}}
\def\name{\mathunderaccent\tilde-3 }

\usepackage{hyperref}

\begin{document}

\title {Graphs represented by Ext}

\author[M. Asgharzadeh]{Mohsen Asgharzadeh}

\address{Mohsen Asgharzadeh, Hakimiyeh, Tehran, Iran.}

\email{mohsenasgharzadeh@gmail.com}

\author[M.  Golshani]{Mohammad Golshani}

\address{Mohammad Golshani, School of Mathematics, Institute for Research in Fundamental Sciences (IPM), P.O.\ Box:
	19395--5746, Tehran, Iran.}

\email{golshani.m@gmail.com}
\urladdr{http://math.ipm.ac.ir/~golshani/}

\author[S. Shelah] {Saharon Shelah}
\address{Einstein Institute of Mathematics\\
Edmond J. Safra Campus, Givat Ram\\
The Hebrew University of Jerusalem\\
Jerusalem, 91904, Israel\\
 and \\
 Department of Mathematics\\
 Hill Center - Busch Campus \\
 Rutgers, The State University of New Jersey \\
 110 Frelinghuysen Road \\
 Piscataway, NJ 08854-8019 USA}
\email{shelah@math.huji.ac.il}
\urladdr{http://shelah.logic.at}
\thanks{The second author's research has been supported by a grant from IPM (No. 1400030417) and Iran National Science Foundation (INSF)  (No. 98008254). The
	third author's research  partially supported by NSF grant no: DMS 1833363. This is publication 1217 of third author.}

\subjclass[2020]{Primary: 03C60;  20A15 Secondary:13L05. }

\keywords {Abelian groups; almost-free modules; Ext-groups; forcing; graph theory; set theoretic methods in  algebra; vanishing of Ext.}

\begin{abstract}
	This paper  
opens and discusses the question originally due to
 Daniel Herden, who asked for which graph $(\mu,R)$
  we can find a family  $\{\mathbb G_\alpha: \alpha < \mu\}$
   of  abelian groups such
that for each $\alpha,\beta\in\mu$
\begin{center}
 $ \Ext(\mathbb G_\alpha, \mathbb G_\beta) = 0$ iff $ (\alpha,\beta) \in R$.
 \end{center}
In this regard, we present  four results. First, we give a connection to
Quillen's small object argument which helps $\Ext$ vanishes and uses to  present a useful criteria to the question. Suppose $\lambda = \lambda^{\aleph_0}$ and $\mu = 2^\lambda$.
We apply Jensen's diamond principle along with the criteria to present $\lambda$-free    abelian groups   representing 
 bipartite graphs.
 Third, we use a version of  black box  to construct  in ZFC, a family of $\aleph_1$-free abelian groups representing bipartite graphs.
 Finally,  applying forcing techniques, we present  a consistent positive answer for general graphs.
\end{abstract}

\maketitle
\numberwithin{equation}{section}\tableofcontents
\section{Introduction}

The vanishing and non-vanishing properties of $\Ext(-,\sim)$ are  useful tools,  see for instance the book \cite{w}. 
Here, we assume the objects $-,\sim$  are not necessarily  noetherian, so the corresponding Ext-family becomes more mysterious.
Despite
its ubiquity, there is very little known about the correspondence between graph theory and the Ext-family.  Our aim in this paper is to present a sample of
 such connection by coding graphs using the vanishing property of Ext of a family of  almost free  abelian groups.

  Recall the following  achievements from literature.
In his seminal paper \cite{Sh:44}, Shelah proved that freeness of   Whitehead groups, that  is an abelian  $\mathbb G$ satisfying the vanishing property $\Ext(\mathbb G,\mathbb{Z})=0$,
is undecidable in ZFC. After this,
there has been a considerable amount of work to understand set theoretical methods in algebra.
G\"obel and
Shelah \cite{gsh} introduced   a  method  to  construct  splitters,  that  is groups $\mathbb G$ satisfying   $\Ext(\mathbb G,\mathbb G)=0$. They applied
their method  to  prove  the existence of enough projective   and injective  objects in the rational cotorsion pairs.
Cotorsion pairs were introduced by Salce \cite{sal} in 1979. Combining with the splitters, this
theory has a lot of applications, not only in group theory but also in the theory of rings and modules.  For instance, see the book \cite{GT}.
G\"obel,
Shelah and Wallutis, proved in \cite{GSW} that any poset embeds into   the lattice of cotorsion pairs of abelian
groups. To be more explicit, let $I$ be any set, and look at the power set $\mathcal{P}(I)$ of $I$. For any
$ X\in \mathcal{P}(I)$ they construct $\aleph_1$-free abelian groups
$\mathbb G_X ,\mathbb H^X$ such that
  for all $X, Y \subseteq I,$
 $$\Ext(\mathbb G_Y, \mathbb H^X) = 0 \Longleftrightarrow Y \subseteq X\quad(\ast).$$
Also, there are some restrictions on the size of 	Ext-groups. As a sample, suppose  $ \mathbb G$ is countable  torsion-free, then $\Ext(\mathbb G,\mathbb{Z})$ is divisible
and hence determined up to isomorphism by its torsion-free rank and its $p$-rank. Shelah   and Str\"ungmann proved that the $p$-rank of $\Ext(\mathbb G,\mathbb{Z})$ is either  countable or  $2^{\aleph_0}$. For this and its generalization, see \cite{ss}.

\begin{definition} Given a directed graph $G=(\mu, R)$, we say that  $G$ is realized as  an Ext-graph of
groups, provided that there exists a family of groups  $\{\mathbb G_\alpha: \alpha < \mu\}$  such
that for each $\alpha,\beta\in\mu$
\begin{center}
 $ \Ext(\mathbb G_\alpha,\mathbb G_\beta) = 0$ iff $ (\alpha,\beta) \in R$.
 \end{center}
\end{definition}
Given the success of representing a wide range of rings as
endomorphism rings of abelian groups (see \cite{EM02} and \cite{GT}),   and useful constructions of  a rigid system consisting of $2^\kappa$
abelian groups (see \cite{fuchs}),
          it is a natural and interesting
question of what can be represented as extension groups   of abelian groups.
In fact,
Daniel Herden asked the following question:

 \begin{question}\label{1.1} Which graphs
$(\mu,R)$ can be realized as an Ext-graph of abelian groups.
\end{question}
Our aim in this paper is to partially answer  Question \ref{1.1}.

The organization of the paper is as follows.
 Section 2 contains the preliminaries and basic notations that we need.
In Section
3 we  apply some ideas similar to
Quillen's small object argument  from model category (see \cite[Theorem 2.1.14]{mark})
 to present a  general criteria for representing bipartite graphs as an Ext-graph of
   $\lambda$-free abelian groups:\mn

\textbf{Theorem (A).}
Let $\mu_1, \mu_2 \leq 2^\lambda$  be such that   $\cf(\mu_2)> \lambda$  and let $R \subseteq \mu_1 \times \mu_2$. Suppose the following assumptions are satisfied:

	\begin{enumerate} \item[(a)]$ \bar{\mathbb{G}} = \langle
		\mathbb{G}_\alpha:\alpha < \mu_1\rangle$ and $\bar{\mathbb{G}}^\iota =
		\langle \mathbb{G}^\iota_\alpha:\alpha < \mu_1\rangle$ for $\iota =1, 2$,
		are sequences of abelian groups,
		\item[(b)] for each $\alpha$,  $\mathbb{G}_\alpha$ is an $\aleph_1$-free abelian
		group of cardinality $\lambda$ and $\mathbb{G}_\alpha =
		\mathbb{G}^2_\alpha/\mathbb{G}^1_\alpha,$ where $\mathbb{G}^1_\alpha \subseteq \mathbb{G}^2_\alpha$ are free abelian groups of
		cardinality $\lambda$,
		\item[(c)]  if $\alpha < \mu_1$ and $\mathbb{L}$ is
		constructible by $\{\mathbb{G}_\gamma:\gamma \in \mu_1 \setminus \alpha\}$
		over $\mathbb{G}^1_\alpha$ then, there is no homomorphism ${\bf g}$ from
		$\mathbb{G}^2_\alpha$ into $\mathbb{L}$ extending $\id_{\mathbb{G}^1_\alpha}$.
	\end{enumerate}

Then there exists a sequence $\bar{\mathbb{K}} = \langle \mathbb{K}_\beta:\beta
< \mu_2\rangle$ equipped with the following three properties:

	\begin{enumerate}
		\item[$(\alpha)$]  $ \mathbb{K}_\beta$ is an $\aleph_1$-free abelian group  of
		cardinality $2^\lambda$ for each $\beta
		< \mu_2$,
		\item[$(\beta)$]   $ \Ext(\mathbb{G}_\alpha, \mathbb{K}_\beta) = 0$ iff
		$\alpha R \beta$,
		\item[$(\gamma)$]   if every $\mathbb{G}_\alpha$ is $\lambda$-free
		then every $\mathbb{K}_\alpha$ is $\lambda$-free as well.
\end{enumerate}

Theorem (A) is one of the main  technical results of the paper and plays an essential role in the sequel. Given  cardinals $\mu < \lambda$ with $\mu$ regular we set
 $$S^\lambda_\mu=\{\alpha < \lambda: \cf(\alpha)=\mu  \}.$$
In Section 4 we apply Jensen's diamond principle $\diamondsuit_S$ along with Theorem (A), and  show
the following:

\textbf{Theorem (B).} Let $S \subseteq
S^\lambda_{\aleph_0}$ be a non-reflecting stationary subset of $\lambda$  and suppose
$\diamondsuit_S$ holds.
Suppose $\lambda = \lambda^{\aleph_0},\mu = 2^\lambda$ and let $R \subseteq \mu
\times \mu$  be a relation. Then
there are sequences $\langle \mathbb G_\alpha: \alpha < \mu \rangle$ and  $\langle \mathbb K_\alpha: \alpha < \mu \rangle$ of $\lambda$-free abelian groups such that for all $\alpha < \mu, |\mathbb G_\alpha|=\lambda,$ $|\mathbb K_\alpha|=2^\lambda$ and for all $\alpha, \beta < \mu,$
$$\Ext(\mathbb G_\alpha, \mathbb K_\beta) = 0 \iff \alpha R \beta.$$

 The
new advantage we have is that we work with $\lambda$-free abelian groups with a control on their size.
 Recall that Jensen's diamond principle is a kind of prediction principle whose truth is independent of ZFC. 
 
In Section 5, we descend from Section 4 to  the ordinary ZFC
set theory  and as another application of  Theorem  (A), we prove the following theorem, where instead of using the diamond principle we use some variant of ``\textit{Shelah's black
	box}'':

\textbf{Theorem (C).}
Let $\lambda = \lambda^{\aleph_0},\mu = 2^\lambda$ and let $R \subseteq \mu
\times \mu$  be a relation. Then
there are families
$\langle  \mathbb G_\alpha: \alpha < \mu \rangle,$ and  $\langle \mathbb K_\alpha: \alpha < \mu \rangle$ of $\aleph_1$-free abelian groups equipped with the following properties:
\begin{enumerate}
	\item  for all $\alpha < \mu, \mathbb G_\alpha$ has size $\lambda$
	and $\mathbb K_\alpha$ has size $2^\lambda,$
	\item for all $\alpha, \beta < \mu,$
	$$\Ext(\mathbb G_\alpha, \mathbb K_\beta) = 0 \iff \alpha R \beta.$$
\end{enumerate}

Here, we lose the $\lambda$-freeness from Theorem (B), the groups are just $\aleph_1$-free,  and this is the price  that Theorem (B) should pay to be in ZFC.
The  black boxes were introduced  by Shelah in \cite{Sh:172},
where he proved that they   can be considered
as  a general method to generate a class of diamond-like principles provable in ZFC.
In particular, Question \ref{1.1} has a positive answer for the case of bipartite graphs,
where: 
\begin{definition}
 A graph $(\mu,R)$ is called bipartite if the vertex
set can be decomposed as $ V_1 \cup V_2 $ such that all edges go between $ V_1$ and
$ V_2$.\end{definition}
 Concerning Theorem (C),   we can realize bipartite  graphs as the Ext-graph of  $\aleph_1$-free abelian groups.
Nevertheless, it is easy to see that bipartite graphs fit in the situation of $(\ast)$.
In particular, we recover the main result of \cite{GSW} by a new argument, see Lemma
\ref{our2}. It may be worth to mention that Theorem (C)  slightly improves \cite{GSW} via  computing the size of objects,  namely  $|\mathbb G_\alpha|=\lambda $ and $|\mathbb K_\alpha|=2^\lambda$ for all $\alpha < \mu$.

In the final section we  prove the following theorem:

\textbf{Theorem (D).}
Suppose GCH holds and the pair $(S, R)$ is a graph  where $R \subseteq S \times S$, and  let $\lambda>|S|$ be an uncountable regular cardinal. Then there exists a cardinal preserving generic extension
of the universe, and there is a family $\{\mathbb{G}_s: s \in S\}$ of $\lambda$-free abelian groups  such that
$$
\Ext(\mathbb{G}_s, \mathbb{G}_t)=0 \iff s R t.
$$ 

The strategy of the proof   of
 Theorem (D) is given in Discussion \ref{dis}. However,
 there are some details to be checked, and this is our task in \S6. Here, the
new advantage we have is that we work with a general graph and also we rely on forcing techniques.
This theorem  gives a consistent positive answer to Question \ref{1.1}. Despite this,  
we think Herden's question has a positive answer in ZFC. Namely,
we
present the following conjecture:

\begin{conjecture}
	Any graph can be realized as  an Ext-graph of groups.
	\end{conjecture}

We hope our results  will shed more light on interplay
between  homological algebra and graph theory.

For all unexplained definitions from algebra
see
the  books by Eklof-Mekler \cite{EM02}  and G\"{o}bel-Trlifaj
\cite{GT}. Also, for  unexplained definitions from the theory of forcing
see the books of Jech \cite{j} and Kunen \cite{k}.

\section{Preliminary notation}
 In this section, we set out our notation and discuss some facts that will be used throughout the paper and refer to the book
 of Eklof and Mekler \cite{EM02} for more information. We restrict our discussion to the category Mod-$\mathbb{Z}$ of  abelian groups, though most of the notions and results can be extended to module categories over more general rings.
For abelian groups $\mathbb{G}$ and $ \mathbb{H}$, we  set $\Ext(\mathbb{G},\mathbb{H}):=\Ext^1_{\mathbb{Z}}(\mathbb{G},\mathbb{H})$ and
similarly, $\Hom(\mathbb{G},\mathbb{H}):=\Hom_{\mathbb{Z}}(\mathbb{G},\mathbb{H})$.
We need the following well-known fact (see e.g. the book \cite[Page 77]{w}):

\begin{fact}\label{yoneda}(Baer and Yoneda)
Let $\zeta_i:=0\longrightarrow \mathbb{B}\stackrel{g_i}\longrightarrow \mathbb{C}_i \stackrel{f_i}\longrightarrow \mathbb{A}\longrightarrow 0$ be two short  exact sequences of abelian groups. We say  $\zeta_1$ is  equivalent to $\zeta_2$ if there is a commutative diagram:
	$$
	\begin{CD}
\zeta_2=	0@>>> \mathbb{B}@>g_2>>\mathbb{C}_2 @>f_2>> \mathbb{A} @>>> 0\\
	@.=@AAA\cong @AAA = @AAA   \\
\zeta_1=	0@>>> \mathbb{B} @>g_1>> \mathbb{C}_1 @>f_1>> \mathbb{A} @>>>0\\
	\end{CD}
	$$
Indeed, this is an equivalent relation, and there is a 1-1 correspondence between the equivalent class
of these short exact sequences and $\Ext(\mathbb{A},\mathbb{B})$. In addition, $[\zeta_1]=0\in \Ext(\mathbb{A},\mathbb{B})$ iff $\zeta_1$
splits.	\end{fact}

\begin{definition}
	\label{k-free}
An abelian group $\mathbb{G}$ is called $\aleph_1$-free if every subgroup of $\mathbb{G}$ of cardinality
$< \aleph_1$, i.e., every countable subgroup, is free. More generally, an abelian group $\mathbb{G}$ is called $\lambda$-free if every subgroup of $\mathbb{G}$ of cardinality
$< \lambda$ is free.
\end{definition}

\begin{definition}
	\label{strongly}
	Let $\kappa$   be a regular cardinal. An abelian group $\mathbb{G}$ is said to be
	strongly $\kappa$-free if there is a set $\mathcal{S}$ of $<\kappa$-generated free subgroups of
	$\mathbb{G}$ containing {0}  such that for any subset $S$ of $\mathbb{G}$ of cardinality
$<\kappa$   and any $\mathbb{N}\in\mathcal{S}$, there is $\mathbb{L}\in\mathcal{S}$  such that $S\cup \mathbb{N}\subset \mathbb{L} $  and
	$\mathbb{L}/ \mathbb{N}$ is free.
\end{definition}

Also, by
a club subset of an uncountable regular cardinal $\kappa$ we mean a closed and unbounded
subset of $\kappa$.

\begin{definition}
Suppose $\kappa$ is an uncountable regular cardinal.
 Let
 $\cD_\kappa$ denote the club filter on $\kappa$, i.e.,
  \begin{center}
$\cD_\kappa=\{A \subseteq \kappa$:  $A$ contains a club  subset of $\kappa$\}.
   \end{center}
      Let also $\mathcal{P}(\kappa)/\cD_\kappa$ denote  the resulting quotient Boolean algebra.
\end{definition}
It is easily seen that $\cD_\kappa$ is a normal $\kappa$-complete filter on $\kappa$ and that it is closed under diagonal intersections, i.e., if
  $A_i \in \cD_\kappa$, for $i<\kappa$, then their diagonal intersection
  $$\bigtriangleup_{i<\kappa}A_i=\{\xi<\kappa: \forall i<\xi, \xi \in A_i      \}$$
  is also in $\cD_\kappa$.
  This can be used to prove the following easy lemma.
 \begin{lemma}
	\label{a15}
	Suppose $\kappa$ is a regular uncountable cardinal, $\delta \leq \kappa$ and   let $\Gamma_i \in \cP(\kappa)/\cD_\kappa$ for $i < \delta$.  Then in the Boolean Algebra
	$\cP(\kappa)/\cD_\kappa,$ the sequence $\{\Gamma_i:i < \delta\}$ has a lub (least upper bound) $\Gamma$.
\end{lemma}
\begin{proof}
For each $i<\delta$ let $A_i \subseteq \kappa$ be such that
 $\Gamma_i = A_i/\cD_\kappa$. If $\delta < \kappa,$ then $\Gamma = A/ \cD_\kappa$ is as required where
	$A=\bigcup\limits_{i<\delta} A_i$, and if $\delta=\kappa$, then $\Gamma = A/\cD_\kappa$
	is as required where $A=\bigtriangleup_{i<\kappa} A_i$ is the diagonal intersection of the sets $A_i, i<\kappa.$
\end{proof}

The following definition plays an important  role in the sequel.
\begin{definition}
	\label{filter}
	Let $\kappa$   be a regular cardinal. If $\mathbb{G}$ is a $\leq\kappa$-generated abelian group, a $\kappa$-filtration of $\mathbb{G}$
	is a sequence $\{\mathbb{G}_{\nu}:\nu < \kappa\}$ of subgroups of $\mathbb{G}$ whose union is $\mathbb{G}$
	 such that for all $\nu<\kappa:$
\begin{enumerate}
	\item[$(a)$] $\mathbb{G}_{\nu}$ is a $< \kappa$-generated subgroup of $\mathbb{G}$;
	\item[$(b)$] if $\mu< \nu$, then $ \mathbb{G}_{\mu}\subset \mathbb{G}_{\nu}$;
	\item[$(c)$] if $\nu$ is a limit ordinal, then $ \mathbb{G}_{\nu}=\bigcup_{\mu<\nu}\mathbb{G}_{\mu} $ i.e., the sequence is continuous.
\end{enumerate}	
\end{definition}

 It is easily seen that if $\{\mathbb{G}_{\nu}:\nu < \kappa\}$  and $\{\mathbb{H}_{\nu}:\nu < \kappa\}$ are two $\kappa$-filtrations  of a
group $\mathbb{G}$, then the set
\[
\{\nu < \kappa: \mathbb{G}_\nu=\mathbb{H}_\nu    \}
\]
contains a club  subset of $\kappa,$ in particular, modulo the club filter $\cD_\kappa,$
the choice of the $\kappa$-filtration does not matter. This observation makes the following definition well-defined.
\begin{definition}
	\label{a13}
	Let $\lambda$ be an uncountable regular cardinals.
\begin{enumerate}	
	\item
	 If $\mathbb{G}$ is an abelian group of cardinality $\lambda$ and $\langle
	\mathbb{G}_\alpha:\alpha < \lambda\rangle$ is a filtration of $\mathbb{G}$, then
	\[
	\Gamma(\mathbb{G},\bar{\mathbb{G}}) = \{\delta < \lambda:\mathbb{G}/\mathbb{G}_\delta \text{ is not
	} \lambda\text{-free}\}.
	\]
	\item Let	
$$\Gamma(\mathbb{G}) = \Gamma(\mathbb{G},\bar{\mathbb{G}})/\cD_\lambda$$
 for some (and hence every)
	filtration $\bar{\mathbb{G}}$ of $\mathbb{G}.$
\end{enumerate}
\end{definition} We recall that $\Gamma(\mathbb{G})$ is called the $\Gamma$-invariant of the
group $\mathbb{G}$, and refer to \cite[ §IV.1] {EM02} for more details and properties of this invariant.
The following lemma  gives a combinatorial characterization for $\lambda$-free groups to be free.
\begin{lemma} (\cite[\text{Ch} IV, Proposition 1.7]{EM02})
\label{comch}
Let $\lambda$ be an uncountable regular cardinals and let $\mathbb G$ be a $\lambda$-free  abelian
group of cardinality $\lambda.$ The following are equivalent:
\begin{enumerate}
\item
 $\mathbb G$ is free,

\item  $\mathbb G$ has a filtration $\langle
	\mathbb{G}_\alpha:\alpha < \lambda\rangle$ such that for all $\alpha < \lambda, \mathbb G_{\alpha+1}/\mathbb{G}_\alpha$ is free,

\item $\Gamma(\mathbb{G})=\emptyset/ \mathcal D_\lambda$.
\end{enumerate}These mean that $\Gamma(\mathbb{G},\bar{\mathbb{G}})$ is non-stationary for some (and hence every)
filtration $\langle
\mathbb{G}_\alpha:\alpha < \lambda\rangle$ of $\mathbb G.$
\end{lemma}

\section{A realization criteria}

Our main result in this section is Theorem \ref{a2}. Let us start by some lemmas and definitions.

\begin{lemma}\label{aleffree}(See \cite[Ex. IV.22]{EM02})
	If $\mathbb{G}$ is the union of a continuous chain $\{\mathbb{G}_\alpha : \alpha \leq \beta\}$ of abelian
	groups such that $\mathbb{G}_0$ and $\mathbb{G}_{\alpha
		+1}/\mathbb{G}_{\alpha}$ are $\aleph_1$-free for all $\alpha + 1 < \beta$,
	then $\mathbb{G}$ is $\aleph_1$-free.
\end{lemma}

More generally, the following
holds:

\begin{lemma}\label{ff} (See \cite[Pages 112-113]{fuchs})
	Let $\kappa<\lambda$ be infinite cardinals. 
\begin{enumerate}
\item[$(i)$] $\lambda$-free implies $\kappa$-free. 
\item[$(ii)$] Subgroups and direct sums of $\kappa$-free are $\kappa$-free 
\item[$(iii)$]   Extension of $\kappa$-free group is $\kappa$-free.
		\item[$(iv)$] Let   $0=\mathbb{G}_0\subseteq \ldots \subseteq \mathbb{G}_i\subseteq \ldots$  be a smooth chain of groups with union $\mathbb{G}$ such
 	that all   $ \mathbb{G}_i /\mathbb{G}_{i+1}$ are $\kappa$-free.  Then $\mathbb{G}$ is   $\kappa$-free.
	\end{enumerate}
\end{lemma}

\begin{definition}
	\label{a1d}Let $\cG$ be a set or class of abelian groups.
	\begin{enumerate}
		\item  We say $\bar{\mathbb{L}}$ is a
		construction by $\cG$ over $\mathbb{G}$ when:
		\begin{enumerate}
			\item[$(a)$]  $\bar{\mathbb{L}} = \langle \mathbb{L}_\varepsilon:\varepsilon \le
			\varepsilon(*)\rangle$ is a $\subseteq$-increasing and continuous sequence of abelian groups,
			\item[$(b)$]  $\mathbb{L}_0 = \mathbb{G}$,
			\item[$(c)$]  for every $\varepsilon < \varepsilon(*),\mathbb{L}_{\varepsilon
				+1}/\mathbb{L}_\varepsilon$ is free or is isomorphic to some member of $\cG$.
		\end{enumerate}
		\item Omitting ``over $\mathbb{G}$" means for $\mathbb{G}=\{0\}$.

		\item We say $\mathbb{L}$ is constructible by $\cG$ (over $\mathbb{G}$) when for some
		$\bar{\mathbb{L}} = \langle \mathbb{L}_\varepsilon:\varepsilon \le
		\varepsilon(*)\rangle$,   $\bar{\mathbb{L}}$ is a construction by $\cG$ (over
		$\mathbb{G}$) and $\mathbb{L}=\mathbb{L}_{\varepsilon(*)}$.
	\end{enumerate}
\end{definition}

\begin{notation}
	Suppose we have the following data of abelian groups and homomorphisms:
	$$
	\begin{CD}
	\mathbb{A}\\
	@AAA  \\
	\mathbb{B} @>>> \mathbb{C}, \\
	\end{CD}
	$$
	We denote the corresponding pushout by
	$\mathbb{A}\oplus_{\mathbb{B}} \mathbb{C}$.
\end{notation}

The next result gives sufficient conditions for representing bipartite graphs using the functor Ext.
\begin{theorem}
\label{a2}
Let $\mu_1, \mu_2 \leq 2^\lambda$  be such that   $ \cf(\mu_2)> \lambda$  and let $R \subseteq \mu_1 \times \mu_2$. Suppose the following assumptions are satisfied:
\mn
\begin{enumerate}
\item[$\boxtimes^1_{\lambda,\mu_1}$]
\begin{enumerate} \item[(a)]$ \bar{\mathbb{G}} = \langle
\mathbb{G}_\alpha:\alpha < \mu_1\rangle$ and $\bar{\mathbb{G}}^\iota =
\langle \mathbb{G}^\iota_\alpha:\alpha < \mu_1\rangle$ for $\iota =1, 2$,
are sequences of abelian groups,
\item[(b)] for each $\alpha$,  $\mathbb{G}_\alpha$ is an $\aleph_1$-free abelian
group of cardinality $\lambda$ and $\mathbb{G}_\alpha =
\mathbb{G}^2_\alpha/\mathbb{G}^1_\alpha,$ where $\mathbb{G}^1_\alpha \subseteq \mathbb{G}^2_\alpha$ are free abelian groups
cardinality $\lambda$,
\item[(c)]  if $\alpha < \mu_1$ and $\mathbb{L}$ is
  constructible by $\{\mathbb{G}_\gamma:\gamma \in \mu_1 \setminus \alpha\}$
  over $\mathbb{G}^1_\alpha$, then there is no homomorphism ${\bf g}:
  \mathbb{G}^2_\alpha\to\mathbb{L}$ extending $\id_{\mathbb{G}^1_\alpha}$:
  $$\xymatrix{
  	&&  \mathbb{G}^1_\alpha \ar[r]^{\subseteq}\ar[d]_{\id}&\mathbb{G}^2_\alpha\ar[d]^{\nexists{\bf g}}\\
  	&& \mathbb{G}^1_\alpha \ar[r]^{\subseteq}& \mathbb{L}
  	&&&}$$

  \end{enumerate}
\end{enumerate}
Then there exists a sequence $\bar{\mathbb{K}} = \langle \mathbb{K}_\beta:\beta
< \mu_2\rangle$ equipped with the following three properties:
\begin{enumerate}
	\item[$\boxtimes^2_{\lambda,\mu_1}$]
\begin{enumerate}
	\item[$(\alpha)$]  $ \mathbb{K}_\beta$ is an $\aleph_1$-free abelian group  of
	cardinality $2^\lambda$ for each $\beta
	< \mu_2$,
	\item[$(\beta)$]   $ \Ext(\mathbb{G}_\alpha, \mathbb{K}_\beta) = 0$ iff
	$\alpha R \beta$,
	\item[$(\gamma)$]   if every $\mathbb{G}_\alpha$ is $\lambda$-free
	then every $\mathbb{K}_\alpha$ is $\lambda$-free as well.
\end{enumerate}\end{enumerate}
\end{theorem}

\begin{proof}
Let $\langle \cU_\varepsilon:\varepsilon < 2^\lambda\rangle$ be a
partition of $2^\lambda$ such that $\cU_\varepsilon \subseteq
[\varepsilon,2^\lambda)$ has cardinality $2^\lambda$.  We  claim that there are
sequences
$\bar{\mathbb{K}}_\varepsilon$,   $\bar{\mathbf H}_{\varepsilon, \beta}$ and  $\bar{\mathbf h}^*_{\varepsilon, \beta}$, for
$\varepsilon \le 2^\lambda$ and $\beta < \mu_2$ with the following properties:
\mn
\begin{enumerate}
\item[$\oplus_{\varepsilon}$]
\begin{enumerate}
\item[(a)] $\bar{\mathbb{K}}_\varepsilon =
\langle \mathbb{K}_{\varepsilon, \beta}:\beta < \mu_2\rangle$ is a sequence of
abelian groups,
\item[(b)]   for each $\beta < \mu_2$ the sequence
$\langle \mathbb K_{\zeta,\beta}:\zeta \le 2^\lambda \rangle$ is
$\subseteq$-increasing and continuous,
\item[(c)]  $ \mathbb{K}_{\varepsilon,\beta}$ has cardinality
$\le 2^\lambda$,

\item[(d)]  $ \bar{\bf H}_{\varepsilon,\beta} = \langle
(\alpha_{\beta,\zeta},h_{\beta,\zeta}):\zeta \in \cU_\varepsilon
\rangle$ lists all the pairs $(\alpha, h)$ where
 $\alpha < \mu_1,~\alpha R \beta$ and $h \in
\Hom(\mathbb{G}^1_\alpha, \mathbb{K}_{\varepsilon,\beta})$,

\item[(e)] $ \bar{\bf h}^*_{\varepsilon,\beta} = \langle
h^*_{\beta,\zeta}:\zeta \in \cU_\varepsilon
\rangle$,
\item[(f)]   if  $\varepsilon \le
\zeta \in \cU_\varepsilon$, then $h^*_{\beta, \zeta} \in \Hom
(\mathbb{G}^2_{\alpha_{\beta,\zeta}},\mathbb{K}_{\zeta +1,\beta})$ extends
 $h_{\beta,\zeta}$.
 The property is conveniently summarized by the subjoined  diagram:

 $$\xymatrix{
 	&&  \mathbb{G}^1_{\alpha_{\beta,\zeta}}  \ar[r]^{\subseteq}\ar[d]_{h_{\beta,\zeta}}&\mathbb{G}^2_{\alpha_{\beta,\zeta}}\ar[d]^{h^*_{\beta, \zeta}}\\
 	&&  \mathbb{K}_{\zeta ,\beta}\ar[r]^{\subseteq}&\mathbb{K}_{\zeta +1,\beta}
 	&&&}$$

\end{enumerate}
\end{enumerate}
\mn
We proceed by a double induction on $\zeta$ and $\beta$ to construct such a sequence.   For $\zeta = 0$ and for all $\beta < \mu_2$ set
$$\mathbb{K}_{0, \beta}:=\bigoplus\{\mathbb{G}^1_\alpha:\alpha < \mu_1,\neg(\alpha R \beta)\}.$$
By Fact \ref{ff}(ii), $\mathbb{K}_{0, \beta}$ is free, and so $\lambda$-free.

For $\zeta$ a limit ordinal and $\beta < \mu_2$, we set $$\mathbb{K}_{\zeta, \beta}=\bigcup\limits_{\varepsilon < \zeta}\mathbb{K}_{\varepsilon, \beta}.$$
Now suppose  that  the groups $K_{\varepsilon, \gamma}$ are defined for all $\varepsilon < \zeta+1$ and $\gamma < \mu_2$. We define $\mathbb{K}_{\zeta+1, \beta}$ for $\beta<\mu_2$ as follows.

 Let  $ \bar{\bf H}_{\varepsilon,\beta} = \langle
(\alpha_{\beta,\zeta},h_{\beta,\zeta}):\zeta \in \cU_\varepsilon
\rangle$ be as in clause (d).
We look at the following diagram
$$
\begin{CD}
\mathbb{K}_{\zeta,\beta} \\
h_{\beta,\zeta}@AAA  \\
 \mathbb{G}^1_{\alpha_{\beta,\zeta}} @>\subseteq>> \mathbb{G}^2_{\alpha_{\beta,\zeta}}, \\
\end{CD}
$$
and  set $$\mathbb{K}_{\zeta+1,\beta}:=\mathbb{K}_{\zeta,\beta}\oplus_{\mathbb{G}^1_{\alpha_{\beta,\zeta}}} \mathbb{G}^2_{\alpha_{\beta,\zeta}}.$$ In particular,  there is a homomorphism $f$ induced from $h_{\beta,\zeta}$  which commutes the following diagram:

$$
\begin{CD}
0@>>>\mathbb{K}_{\zeta,\beta} @>g>>\mathbb{K}_{\zeta+1,\beta} @>>> \frac{\mathbb{G}^2_{\alpha_{\beta,\zeta}}}{\mathbb{G}^1_{\alpha_{\beta,\zeta}}} @>>> 0\\
   @.h_{\beta,\zeta}@AAAf@AAA = @AAA   \\
0@>>> \mathbb{G}^1_{\alpha_{\beta,\zeta}} @>>> \mathbb{G}^2_{\alpha_{\beta,\zeta}} @>>> \frac{\mathbb{G}^2_{\alpha_{\beta,\zeta}}}{\mathbb{G}^1_{\alpha_{\beta,\zeta}}} @>>>0\\
\end{CD}
$$
We set $h^*_{\beta, \zeta}:=f$. By a diagram chasing,  $h^*_{\beta, \zeta}$ extends
$h_{\beta,\zeta}$, and recall that $g:\mathbb{K}_{\zeta,\beta} \hookrightarrow\mathbb{K}_{\zeta+1,\beta}$ is an embedding. This completes the inductive construction and hence proves the existence of the sequences claimed to exist above, i.e., the construction of $\oplus_{\varepsilon}$ is now complete. Suppose now that every  group $\mathbb{G}_\alpha$ is $\lambda$-free. By Fact \ref{ff}(iii)
it follows that $\mathbb{K}_{\zeta,\beta}$ is $\lambda$-free as well.

For $\beta < \mu_2$ let
$$\mathbb{K}_\beta := \mathbb{K}_{2^\lambda,\beta}.$$
In view of Lemma \ref{aleffree}, the above short exact sequence and the hypotheses, each $\mathbb{K}_\beta$ is an $\aleph_1$-free abelian group of size $2^\lambda.$ Recall that  $\lambda$-freeness of
$\mathbb{K}_\alpha$ follows by $\lambda$-freeness of
$\mathbb{G}_\alpha$, as desired by $\boxtimes^2_{\lambda,\mu_1}(\gamma)$.

Let us check the item $\boxtimes^2_{\lambda,\mu_1}(\beta)$.
 First, suppose that
$\alpha R \beta$ and let $h \in
\Hom (\mathbb{G}^1_\alpha, \mathbb{K}_\beta)$. Since $|\mathbb{G}^1_\alpha| \le \lambda <
\cf(\mu_2)$, for some $\varepsilon < 2^\lambda$ we have
$\Rang(h)\subseteq K_{\varepsilon,\beta}$, and hence we can
assume that
 $h
\in \Hom (\mathbb{G}^1_\alpha, \mathbb{K}_{\varepsilon,\beta})$. By the definition of $\bar{\bf H}_{\varepsilon,\beta},$  for some
$\zeta \in  \cU_\varepsilon$ we have $(\alpha, h) =
(\alpha_{\beta,\zeta},h_{\beta,\zeta})$. Then by clause (f) of the construction, $h^*_{\beta,\zeta} \in
\Hom(\mathbb{G}^2_\alpha,\mathbb{K}_{\zeta +1,\beta})$ extends
$h_{\beta,\zeta}$. There is a natural embedding map from $\Hom(\mathbb{G}^2_\alpha,\mathbb{G}_{\zeta +1,\beta})$
 into $\Hom (\mathbb{G}^2_\alpha, \mathbb{K}_\beta)$, so without loss of generality, we may and do assume that $$h^*_{\beta,\zeta} \in
\Hom (\mathbb{G}^2_\alpha, \mathbb{K}_\beta).$$ Now, we look at the exact sequence
$$(+):=0\longrightarrow \mathbb{G}^1_\alpha \longrightarrow \mathbb{G}^2_\alpha\longrightarrow \mathbb{G}_\alpha\longrightarrow 0.$$
Applying $\Hom(-,\mathbb{K}_\beta)$ to $(+)$, it induces the following long exact sequence

 $$\Hom(\mathbb{G}^2_\alpha, \mathbb{K}_\beta)\stackrel{{\bf f}}\longrightarrow \Hom( \mathbb{G}^1_\alpha, \mathbb{K}_\beta)\longrightarrow\Ext(\mathbb{G}_\alpha,\mathbb{K}_\beta)\longrightarrow\Ext(\mathbb{G}^2_\alpha, \mathbb{K}_\beta) =0,$$
where the last vanishing $\Ext(\mathbb{G}^2_\alpha, \mathbb{K}_\beta) =0$ holds as  $\mathbb{G}^2_\alpha $ is free. From this, $$\Ext(\mathbb{G}_\alpha,\mathbb{K}_\beta)=\frac{\Hom( \mathbb{G}^1_\alpha,\mathbb{K}_\beta)}{\Rang(f)}.$$
Thus, $\Ext(\mathbb{G}_\alpha, \mathbb{K}_\beta)=0$ if and only if ${\bf f}$ is surjective.
But as we observed above, ${\bf f}$ is onto, and hence  we conclude that $\Ext(\mathbb{G}_\alpha, \mathbb{K}_\beta)=0$.

Let $(\alpha,\beta) \in (\mu_1 \times \mu_2)
\setminus R$. We need to show $\Ext(\mathbb{G}_\alpha, \mathbb{K}_\beta) \ne 0$.
Suppose on the contrary that $\Ext(\mathbb{G}_\alpha, \mathbb{K}_\beta) = 0$, and  search for a contradiction.
Since $\mathbb{G}^1_\alpha$ is a direct summand of $\mathbb{K}_{0,\beta}$,  we have $ \mathbb{G}^1_\alpha\oplus X= \mathbb{K}_{0,\beta}$. Let   $\rho:\mathbb{G}^1_\alpha\to\mathbb{G}^1_\alpha\oplus X $ be the natural map, and look at$$
\begin{CD}
g^1_\alpha:=\mathbb{G}^1_\alpha@>\rho>>\mathbb{G}^1_\alpha\oplus X@>=>> \mathbb{K}_{0,\beta}  @>\subseteq>>  \mathbb{K}_\beta
\end{CD}
$$

 So  by our assumption, there is
$g^2_\alpha: \mathbb{G}^2_\alpha \rightarrow \mathbb{K}_\beta$ extending $g^1_\alpha$:$$\xymatrix{
	&&  \mathbb{G}^1_\alpha \ar[r]^{\subseteq}\ar[d]_{g^1_\alpha}&\mathbb{G}^2_\alpha\ar[dl]^{\exists g^2_\alpha}\\
	 &&  \mathbb{K}_\beta
	&&&}$$

We now show  that there exists $u \subseteq 2^\lambda$  equipped with:
\begin{enumerate}
	\item[(a)]    $|u| =
	\lambda$,
	\item[(b)]   $ 0 ,\alpha,\beta \in u$,
	
	\item[(c)]  	$\varepsilon \in u \Rightarrow \Rang(h_{\beta, \varepsilon}) \subseteq
	\sum\limits_{\zeta \in u \cap \varepsilon} \Rang(h^*_{\beta, \zeta})+
	\sum\limits_{\gamma \in u} \mathbb{G}^1_\gamma$,
	\item[(d)]   $\Rang(g^2_\alpha)
	\subseteq	\sum\limits_{\varepsilon \in u}\Rang(h^*_{\beta, \varepsilon}) +
	\sum\limits_{\gamma \in u} \mathbb{G}^1_\gamma$.
	\end{enumerate}
Indeed, by induction on $i <\lambda$ we define an increasing and continuous sequence $\langle
u_i:i<\lambda\rangle$ of subsets of $2^{\lambda}$ such that  for each
$i <\lambda$, the following properties are valid:

\begin{enumerate}
\item[$\bullet$]   $|u_i| \leq
	\lambda$,
	
\item[$\bullet$]    $ 0 ,\alpha,\beta \in u_0$,

\item[$\bullet$]  If $i$ is limit ordinal, then $u_i=\bigcup_{j<i} u_j $,

\item[$\bullet$]  $\Rang(g^2_\alpha)
	\subseteq \sum\limits_{\varepsilon \in u_0}\Rang(h^*_{\beta, \varepsilon}) +
	\sum\limits_{\gamma \in u_0} \mathbb{G}^1_\gamma$,
	
\item[$\bullet$]      If $\varepsilon \in u _i$ then  $$ \Rang(h_{\beta, \varepsilon}) \subseteq
	\sum\limits_{\zeta\in u_{i+1}\cap \varepsilon}  \Rang(h^*_{\beta, \zeta})+
	\sum\limits_{\gamma \in u_{i+1}} \mathbb{G}^1_\gamma.$$
\end{enumerate}
Set $u:=\bigcup\limits_{i<\lambda}u_i$. It is easily seen that $u$ satisfies  the required properties.
Let $\varepsilon(*) = \otp(u)$, where $\otp(u)$ denotes the order type of the set $u$ and let $\langle
\gamma_\varepsilon:\varepsilon < \varepsilon(*)\rangle$ be the increasing enumeration of $u$.

Set
$$\cG:=\langle \mathbb{G}_\gamma: \gamma\in u\setminus\alpha\rangle.$$
We now define a
construction  $\bar{\mathbb{L}}:= \langle \mathbb{L}_\varepsilon:\varepsilon \le
\varepsilon(*)\rangle$ by $\cG$  over $\mathbb{L}_0=\mathbb{G}_\alpha^1$. To achieve this, we set $$ \mathbb{L}_{\epsilon+1}:= 	 \sum\limits_{\zeta<\epsilon} \Rang(h^*_{\beta, \zeta})+
\sum\limits_{\gamma \in u} \mathbb{G}^1_\gamma.$$Let $A:=\cok( h^*_{\beta, \zeta})$. Let put
this into the previous diagram and obtain the following: $$
\begin{CD} @. 0 @.  0 \\@.@AAA   @AAA    \\
  @. A @>=>> A \\@.\pi@AAA\pi   @AAA    \\
0@>>>\mathbb{K}_{\zeta,\beta} @>>>\mathbb{K}_{\zeta+1,\beta} @>>> \mathbb{G}_{\alpha_{\beta,\epsilon}} @>>> 0\\
@.h_{\beta,\zeta}@AAA  h^*_{\beta, \zeta}@AAA = @AAA   \\
0@>>> \mathbb{G}^1_{\alpha_{\beta,\zeta}} @>>> \mathbb{G}^2_{\alpha_{\beta,\zeta}} @>>> \mathbb{G}_{\alpha_{\beta,\epsilon}} @>>>0\\
\end{CD}
$$This yields that:$$
\begin{CD}
0@>>>\Rang (h^*_{\beta, \zeta} )  @>>>\mathbb{K}_{\zeta+1,\beta} @>>> A @>>> 0\\
@.f_1@AAA  f_2@AAA = @AAA   \\
0@>>>\Rang (h_{\beta, \zeta} )@>>> \mathbb{K}_{\zeta,\beta} @>>> A @>>>0,\\
\end{CD}
$$where $f_i$ are the natural inclusions. Now, we use  the $\Ker-\cok$ exact sequence:

$$0\to \Ker(f_1)\to \Ker(f_2)\to \Ker(=)\to \cok(f_1)\to \cok(f_2)\to \cok(=)\to 0$$
Recall that $\Ker(=)= \cok(=)=0$, $ \cok(f_1)=\frac{\Rang (h^*_{\beta, \zeta} )}{\Rang (h_{\beta, \zeta} )}$ and $ \cok(f_2)= \mathbb{G}_{\alpha_{\beta,\epsilon}}$. So, $$\Rang (h^*_{\beta, \zeta} )/ \Rang (h_{\beta, \zeta} ) \cong  \mathbb{G}_{\alpha_{\beta,\epsilon}}\quad(+)$$
For each $\epsilon <\epsilon(\star)$ we have $$\frac{\mathbb{L}_{\epsilon+1}}{\mathbb{L}_{\epsilon}}\cong\Rang (h^*_{\beta, \zeta} )/ \Rang (h_{\beta, \zeta} )  \cong  \mathbb{G}_{\alpha_{\beta,\epsilon}}\in \cG.$$ In view of Definition \ref{a1d},
$\bar{\mathbb{L}}$ is indeed a construction.
Recall from $(d)$  that $\Rang(g^2_\alpha)\subset \mathbb{L}_{\varepsilon(*)}$. Hence  $ g^2_\alpha:\mathbb{G}^2_\alpha\to \mathbb{L}_{\varepsilon(*)}$ extends the identity function over $\mathbb{G}^1_\alpha$. This is in contradiction with $\boxtimes^1_{\lambda,\mu_1}(c)$.
\end{proof}
\section{From  diamond to Ext-graph of quite free groups }
In this section we are going to prove Theorem (B) from the introduction.

\begin{notation}
 	Let $\mathbb{G}$ be a reduced torsion free abelian group, e.g., a free group.
	The notation $ {\mathbb{G}}$ stands for the  $\mathbb{Z}$-adic completion of $\mathbb{G}$.\end{notation}

\begin{remark} i) $\mathbb{Z}$-adic topology of a  free abelian group is Hausdorff.
	
	 ii) For example, each element of  $\widehat{\mathbb{\mathbb{Z}}}$ can be
	represented as $\Sigma_n k_n$ where $n! |k_n$.
\end{remark}
\begin{discussion}
Recall that for a stationary set $S \subseteq \lambda,$ Jensen's diamond $\diamondsuit_\lambda(S)$ asserts the existence
of a sequence $\langle  S_\alpha: \alpha \in S   \rangle$
such that for every $X \subseteq \lambda$ the set $\{\alpha \in S: X \cap \alpha=S_\alpha    \}$
is stationary.
\end{discussion}

For simplicity of the reader we cite the following result:

\begin{fact}\label{dec}
 If $\diamondsuit_\lambda(S)$  holds, then there is a decomposition $S =\bigcup_{\beta <\lambda }S_{\beta}$
such that $\diamondsuit_\lambda(S_{\beta})$ holds for all $\beta <\lambda$.
\end{fact}
\begin{proof}
 See \cite[Theorem 9.1.17]{GT}.
\end{proof}

The following lemma plays a key role in our construction.

\begin{lemma}
	\label{nonext}
 Let $\iota\in \{1, 2\}$. Suppose $\mathbb G^\iota_n$  are free abelian groups of the same size such that for all $n<\omega$  the following conditions are satisfied:
	\begin{enumerate}
		\item $\mathbb{G}^1_n \subseteq \mathbb{G}^2_n,$
		\item $\mathbb{G}^\iota_n \subseteq \mathbb{G}^\iota_{n+1}$,
		\item  $\mathbb{G}^\iota_{n+1}/ \mathbb{G}^\iota_n$ is free,
		
		\item $\mathbb{G}^2_{n+1}/ (\mathbb{G}^1_{n+1}+\mathbb{G}^2_n)$ is free,
		
		\item $\mathbb{G}^2_{n+1}=\left(\mathbb{G}^1_{n+1}\oplus_{\mathbb G^1_n} \mathbb G^2_n\right)\oplus \bigoplus\limits_{m<\omega} y_m\mathbb Z,$
		
		\item $\bar{\mathbb{L}}=\langle \mathbb{L}_\varepsilon: \varepsilon \leq \varepsilon(*)       \rangle$	is a construction,
		
		\item $\mathbb{L}_{\varepsilon}$ and $\mathbb{L}_{\varepsilon+1}/ \mathbb{L}_\varepsilon$ are $\aleph_1$-free.
	\end{enumerate}
	Let $\mathbb G^\iota=\bigcup\limits_{n<\omega}\mathbb{G}^\iota_n,$   so that $\mathbb G^1 \subseteq \mathbb G^2$ are free abelian groups and let
	$\bold f \in \Hom(\mathbb G^2, \mathbb L_{\varepsilon(*)})$.
	Then there are free abelian groups $\dot{\mathbb G}^1  \subseteq \dot{\mathbb{G}}^2$ such that:
	\begin{enumerate}
		\item[(a)] $\mathbb{G}^1 \subseteq \dot{\mathbb G}^1$,
		
		\item[(b)] $\mathbb{G}^2 \subseteq \dot{\mathbb G}^2$,
		
		\item[(c)] $\mathbb{G}^1=\mathbb{G}^2 \cap \dot{\mathbb G}^1$,
		
		\item[(d)] $\dot{\mathbb G}^\iota / {\mathbb G}^\iota$ is free,
		
		\item[(e)] $\dot{\mathbb G}^2/ (\dot{\mathbb G}^1+\mathbb{G}^2)$ is torsion free,
		
		\item[(f)] for all $n<\omega$, $\dot{\mathbb G}^2/ (\dot{\mathbb G}^1+\mathbb{G}^2_n)$ is free,

		\item[(g)]  there are no $\bold{\dot{f}}$ and $ \langle  \dot{\mathbb{L}}_\varepsilon: \varepsilon \leq \varepsilon(*) \rangle$ equipped with the following properties:
		\begin{enumerate}
			\item[$(g1)$] $\langle  \dot{\mathbb{L}}_\varepsilon: \varepsilon \leq \varepsilon(*) \rangle$ is a nice construction over $\dot{\mathbb{G}}^1,$
			
			\item[$(g2)$]  $\bold{\dot{f}} \in \Hom(\dot{\mathbb{G}^2}, \dot{\mathbb{L}}_{\varepsilon(*)})$,
			
			\item[$(g3)$] $\dot{\bold f}$ extends $\bold f \cup \id_{\dot{\mathbb G}^1}$.
		 This property is conveniently summarized by the subjoined  diagram:

		\[
		\xymatrix@=3pc{
			& &  \dot{\mathbb{L}}_{\varepsilon(*)} \\
			\dot{\mathbb G}^1 \ar[r]_{\subseteq}\ar@/^/[rru]^{\subseteq} & \dot{\mathbb G}^2 \ar[ru]_{\nexists\dot{\bold f}}   & \\
			\mathbb{G}^1 \ar [u]^\subseteq \ar[r]_\subseteq & \mathbb{G}^2\ar[u] ^\subseteq \ar@/_/[ruu]_{{\bold f}}  &
		}
		\]

			\item[$(g4)$] $\mathbb{L}_\varepsilon \subseteq \dot{\mathbb{L}}_\varepsilon$ for $\varepsilon \leq \varepsilon(*),$
			
			\item[$(g5)$] $\mathbb{L}_\varepsilon = \mathbb{L}_{\varepsilon+1} \cap \dot{\mathbb{L}}_\varepsilon,$
			
			\item[$(g6)$] $\dot{\mathbb L}_{\varepsilon+1}/ (\mathbb{L}_{\varepsilon+1}+\dot{\mathbb L}_\varepsilon)$ is free.
		\end{enumerate}
	\end{enumerate}
\end{lemma}

\begin{proof}
 	Set $\dot{\mathbb G}^1: = \mathbb G^1 \oplus \bigoplus\limits_{n<\omega} z_n \mathbb{Z}$, where $\bigoplus\limits_{n<\omega} z_n \mathbb{Z}$ is a free abelian group with a base  given by $\{z_n:n<\omega\} $.
	To define $\dot{\mathbb G}^2,$ let us first set
	\[
	\mathbb G^{2, \ast}:= \dot{\mathbb G}^1 \oplus_{\mathbb G^1} \mathbb G^2.
	\]
	For any infinite sequence $\vec{a}=\langle a_n: n<\omega \rangle\in \prod _\omega2=  {}^{\omega}2$, we look at the following
	\[
	\mathbb G^{2, \vec{a}} := \langle \mathbb G^{2, \ast} \cup \{ \sum\limits_{n<\omega} n!(y_n+ a_n z_n)                   \} \rangle,
	\]i.e., the subgroup of $\widehat{\mathbb G^{2, \ast}}$ generated by $\mathbb G^{2, \ast}$ and the distinguished element $$ \sum\limits_{n<\omega} n!(y_n+ a_n z_n) \in\widehat{\mathbb G^{2, \ast}} .$$
	It is routine to see that the properties $(a)-(f)$ are satisfied with $\mathbb G^{2, \vec{a}}$. We are going to show that $\dot{\mathbb G}^2=\mathbb G^{2, \vec{a}}$ satisfies in the remaining  property $(g)$  for a suitable choice of $\vec{a}$. Suppose on the way of contradiction  that for each $\vec{a}=\langle a_n: n<\omega \rangle \in {}^{\omega}2$
	there is a counterexample $\bold{\dot{f}}_{\vec a}$ and $ \langle  \dot{\mathbb{L}}^{\vec a}_\varepsilon: \varepsilon \leq \varepsilon(*) \rangle$ to clause $(g)$. In particular,
	$\bold{\dot{f}}_{\vec a} \in \Hom(\mathbb G^{2, \vec{a}}, \dot{\mathbb{L}}_{\varepsilon(*)})$,
	and it extends $\bold f \cup \id_{\dot{\mathbb G}^1}$.

	For $\vec a \in {}^{\omega}2$ and $m<\omega$ set
	\[
(+) \quad\quad\quad	y^{\vec a}_m:= \sum\limits_{n \geq m}\frac{n!}{m!}(y_n+a_n z_n).
	\]

	\begin{enumerate}
	\item[{\bf Claim}](A): Let $\vec{a} $, $ \vec{b}$ and $m$ be such that
		\begin{enumerate}
	\item[$\bullet$] $m<\omega$,
	\item[$\bullet$] $a_\ell=b_\ell$ for all $\ell\leq m $,
	\item[$\bullet$]	$\bold{\dot{f}}_{\vec a} (y^{\vec a}_\ell)=\bold{\dot{f}}_{\vec b} (y^{\vec b}_\ell)$ for all $\ell \leq n$.
	\end{enumerate}
	Then $\vec{a} = \vec{b}$.
\end{enumerate}	
To prove the claim,
 we proceed by induction on $\ell < \omega$ and show that:
		\begin{enumerate}
	\item[(i)] $a_{\ell}=b_{\ell}$,
	\item[(ii)] 	$\bold{\dot{f}}_{\vec a} (y^{\vec a}_\ell)=\bold{\dot{f}}_{\vec b} (y^{\vec b}_\ell)$.
\end{enumerate}
The case $\ell \leq m$ follows from the assumption, so we may assume that $\ell > m$.
 Now suppose that $\ell =k+1$ and the result is true for all natural numbers less or equal to
$k$. Let us evaluate  $\bold{\dot{f}}_{\vec a}$ at $y^{\vec a}_m$ from $(+)$, and recall that $\bold{\dot{f}}_{\vec a}$ extends $\bold f \cup \id_{\dot{\mathbb G}^1}$. It turns out that:$$\bold{\dot{f}}_{\vec a} (	y^{\vec a}_m)= \sum\limits_{n \geq m}(\frac{n!}{m!}\bold{\dot{f}}_{\vec a} (y_n)+a_n\bold{\dot{f}}_{\vec a} (z_n))= \sum\limits_{n \geq m}\frac{n!}{m!} (y_n+a_n z_n).$$

It immediately follows that for all $i<\omega$
$$(\ast)\quad\quad\quad\bold{\dot{f}}_{\vec a} (y^{\vec a}_i)=(i+1)\bold{\dot{f}}_{\vec a} (y^{\vec a}_{i+1})-(y_i+a_iz_i).$$
We  apply $(\ast)$ at level $i=k$ and combine it with inductive hypothesis to observe that

	\begin{equation*}
\begin{array}{clcr}
0&=\bold{\dot{f}}_{\vec a} (y^{\vec a}_k)-\bold{\dot{f}}_{\vec b} (y^{\vec b}_k)\\ &= \ell\bold{\dot{f}}_{\vec a} (y^{\vec a}_{\ell})-(y_{\ell-1}+a_{\ell-1}z_{\ell-1})-\ell\bold{\dot{f}}_{\vec a} (y^{\vec a}_{\ell})+(y_{\ell-1}+a_{\ell-1}z_{\ell-1})\\
&
= \ell(\bold{\dot{f}}_{\vec a} (y^{\vec a}_\ell)-\bold{\dot{f}}_{\vec b} (y^{\vec b}_\ell)).
\end{array}
\end{equation*}	
Since $\dot{\mathbb{L}}_{\varepsilon(*)}$ is torsion-free, it follows that $\bold{\dot{f}}_{\vec a} (y^{\vec a}_\ell)-\bold{\dot{f}}_{\vec b} (y^{\vec b}_\ell)=0$, so (ii) is valid. In order to show  (i) we apply $(\ast)$
at level $i=\ell$:

$$0=\bold{\dot{f}}_{\vec a} (y^{\vec a}_\ell)-\bold{\dot{f}}_{\vec b} (y^{\vec b}_\ell)= (\ell+1)(\bold{\dot{f}}_{\vec a} (y^{\vec a}_{\ell+1})-\bold{\dot{f}}_{\vec b} (y^{\vec b}_{\ell+1}))-(a_{\ell}-b_{\ell})z_{\ell},
$$
i.e.,$$(a_{\ell}-b_{\ell})z_{\ell}=(\ell+1)(\bold{\dot{f}}_{\vec a} (y^{\vec a}_\ell)-\bold{\dot{f}}_{\vec b} (y^{\vec b}_\ell)).$$
It follows  that $(\ell+1)|(a_{\ell}-b_{\ell})$. Since $a_i,b_i\in\{0,1\}$, we must have $a_{\ell}-b_{\ell}=0$.
This completes the proof of the desired claim.
Let $\vec{a}$ and $\vec{b}$ be distinct in ${}^{\omega}2$
such that $a_0=b_0$ and $\bold{\dot{f}}_{\vec a} (y^{\vec a}_0)=\bold{\dot{f}}_{\vec b} (y^{\vec b}_0)$. This contradicts
 Claim(A), so we are done.
\end{proof}

\begin{remark}
Adopt the notation of Lemma	\ref{nonext}, and only  replace ``free'' by ``$\aleph_1$-free''.
Then the same claim, as Lemma	\ref{nonext} indicates, is valid. We leave its straightforward modification   to the reader.
\end{remark}

Recall that  $S^\lambda_{\aleph_0}=\{\alpha < \lambda: \cf(\alpha)=\aleph_0     \}$.
The main point in the development of Theorem B) is   to introduce $\boxtimes_{\lambda,\mu}^1$ from Theorem \ref{a2}. Here, we present a variation of it:

\begin{theorem}
	\label{underdiamond}Let $S \subseteq
	S^\lambda_{\aleph_0}$ be stationary non-reflecting and suppose
	$\diamondsuit_S$ holds. Suppose one of the followings:
\begin{enumerate}
	\item[$(1)$]	 $\mu = \lambda = \cf(\lambda) > \aleph_0$, or
	
	\item[$(2)$]  $\mu = 2^\lambda,\lambda =  \cf(\lambda) > \aleph_0$.
	\end{enumerate}
Let $\iota =1, 2$.
 Then  there are sequences
  $\bar{\mathbb{G}}^\iota =
 \langle \mathbb{G}^\iota_\alpha:\alpha < \mu \rangle$   of abelian groups such that  the following assertions are valid:
\begin{enumerate}	\item[(i)]  $\mathbb{G}^1_\alpha \subseteq \mathbb{G}^2_\alpha$ are free abelian groups
	cardinality $\lambda$,
	\item[(ii)] for each $\alpha$,  $\mathbb{G}_\alpha:=\mathbb{G}^2_\alpha/ \mathbb{G}^1_\alpha$ is a strongly $\lambda$-free abelian
	group of cardinality $\lambda$,
	\item[(iii)]  if $\alpha < \mu $ and $\mathbb{L}$ is
	constructible by $\{\mathbb{G}_\gamma:\gamma \in \mu \setminus \alpha\}$
	over $\mathbb{G}^1_\alpha$, then there is no homomorphism ${\bf g}:
	\mathbb{G}^2_\alpha\to\mathbb{L}$ extending $\id_{\mathbb{G}^1_\alpha}$.
\end{enumerate}
\end{theorem}

\begin{proof}
 (1): Since $S$ is non-reflecting, given any limit ordinal $\delta < \lambda,$ there exists a club $C_\delta \subseteq \delta$ such that
$S \cap C_\delta=\emptyset.$
	Furthermore as $\diamondsuit_S$  holds, and in the light of Fact \ref{dec}, we can find a sequence $\langle S_\varepsilon:\varepsilon < \lambda \rangle$ of
	subsets of $S$ such that the following two properties are satisfied:
	\begin{itemize}
		\item the sets $S_\varepsilon$ are 	pairwise almost
		disjoint, i.e., for all $\varepsilon < \zeta < \lambda, S_\varepsilon \cap S_\zeta$ is a bounded subset of $\lambda,$
		\item	$\diamondsuit_{S_\varepsilon}$ holds  for $\varepsilon < \lambda$.
	\end{itemize}
	For $\varepsilon < \lambda$ let
	\[
	\langle (f^\varepsilon_\alpha, \mathbb{H}^{1, \varepsilon}_\alpha, \mathbb{H}^{2, \varepsilon}_\alpha, \bar{\mathbb{L}}_\alpha): \alpha \in S_\varepsilon                 \rangle
	\]
	be such that:
	\begin{itemize}
		\item $f^\varepsilon_\alpha: \alpha \to \alpha$ is a function,
		
		\item $\mathbb{H}^{1, \varepsilon}_\alpha \subseteq \mathbb{H}^{2, \varepsilon}_\alpha$ are abelian groups,
		
		\item the universe of $\mathbb{H}^{2, \varepsilon}_\alpha$ is $\alpha$,

\item $\bar{\mathbb{L}}_\alpha=\langle \mathbb{L}^\alpha_{\varepsilon}: \varepsilon \leq \alpha \rangle$ is a construction such that $\mathbb{L}_{\alpha}$
has universe $\alpha,$

 \item  each $\mathbb{L}_\varepsilon$
and $\mathbb{L}_{\varepsilon+1}/ \mathbb{L}_\varepsilon$ is $\aleph_1$-free,
		
		\item for any tuple $(f, \mathbb{H}^1, \mathbb{H}^2, \bar{\mathbb{L}})$, where $f: \lambda \to \lambda$ is a function, $\mathbb{H}^1 \subseteq \mathbb{H}^2$
		are abelian groups where the universe of $\mathbb{H}^2$ is $\lambda$ and $\bar{\mathbb{L}}=\langle \mathbb{L}_{\varepsilon}: \varepsilon   \leq \lambda \rangle$ is a construction such that the universe of $\mathbb{L}_\lambda$ is $\lambda$ and $\mathbb{L}_\varepsilon$ and $\mathbb{L}_{\varepsilon+1}/ \mathbb{L}_\varepsilon$ are $\aleph_1$-free, then the following set
		\[
		\{ \alpha \in S_\varepsilon:  (f\restriction \alpha, \mathbb{H}^1 \restriction \alpha, \mathbb{H}^2 \restriction \alpha, \bar{\mathbb{L}} \restriction \alpha)=              (f^\varepsilon_\alpha, \mathbb{H}^{1, \varepsilon}_\alpha, \mathbb{H}^{2, \varepsilon}_\alpha, \bar{\mathbb{L}}_\alpha) \}
		\]
		is stationary.
	\end{itemize}
	Given $\varepsilon< \lambda$ by induction on $\alpha < \lambda$, we choose the sequences
	$\bar{\mathbb{G}}^\iota_\varepsilon = \langle
	\mathbb{G}^\iota_{\alpha,\varepsilon}:\alpha < \lambda\rangle$, for $\iota=1, 2$ such that:
	\begin{enumerate}
		\item[$\circledast$]
		\begin{enumerate}
			\item[$(a)$] $\bar G^\iota_\varepsilon$ is an increasing
			and continuous sequence of abelian groups $\mathbb{G}^\iota_{\alpha,\varepsilon}=(G^\iota_{\alpha, \varepsilon}, +, 0)$.
			
			\item[$(b)$]   $\mathbb{G}^\iota_{0,\varepsilon} = \{0\}$ and for $\alpha < \lambda,$
			the universe of $\mathbb{G}^2_{\alpha,\varepsilon}$, namely $G^2_{\alpha, \varepsilon}$ is an ordinal
			$\gamma_{\alpha,\varepsilon} < \lambda$.
			
			\item[$(c)$] The following holds:
			\begin{enumerate}
				\item[$(c1): $]  $\mathbb{G}^1_{\alpha,\varepsilon}
				\subseteq \mathbb{G}^2_{\alpha,\varepsilon}$ are free.
				\item[$(c2):$] If $\alpha < \beta$ then $
				\mathbb{G}^2_{\alpha, \varepsilon} \cap \mathbb{G}^1_{\beta, \varepsilon} = \mathbb{G}^1_{\alpha, \varepsilon}$.
				\item[$(c3): $]  If $\alpha < \beta$ and $\alpha \notin
				S_\varepsilon$, then  $\mathbb{G}^2_\beta/(\mathbb{G}^1_\beta + \mathbb{G}^2_\alpha)$ is free.
			\end{enumerate}
			\item[$(d)$]If $\alpha < \beta$ and $\alpha \notin S_\varepsilon$,
			then $\mathbb{G}^1_{\beta,\varepsilon}/\mathbb{G}^1_{\alpha,\varepsilon}$
			and $\mathbb{G}^2_{\beta,\varepsilon}/\mathbb{G}^2_{\alpha,\varepsilon}$ are
			free.
			
			\item[$(e)$] if $\delta \in S_\varepsilon,$ then there are no $\dot{\bold f}$ and $ \langle  \dot{\mathbb{L}}_\zeta: \zeta \leq \delta \rangle$ equipped with the following properties:
\begin{enumerate}
\item[$(e1)$] $\langle  \dot{\mathbb{L}}_\zeta: \zeta \leq \delta \rangle$ is a   construction over $\mathbb{G}_{\delta+1, \varepsilon}^1,$

\item[$(e2)$]  $\dot{\bold f} \in \Hom(\mathbb{G}^2_{\delta+1, \varepsilon}, \dot{\mathbb{L}}_{\delta})$,

\item[$(e3)$] $\dot{\bold f}$ extends $\bold{f}^\varepsilon_\delta \cup \id_{\mathbb{G}_{\delta+1, \varepsilon}^1}$. This means that $\mathbb{G}^2_{\delta, \varepsilon}\subseteq \mathbb{H}^{2, \varepsilon}_{\delta}$,  and also:

\[
\xymatrix@=3pc{
	& & \dot{\mathbb{L}}_{\delta} \\
	\mathbb{G}^1_{\delta+1, \varepsilon} \ar[r]_\subseteq\ar@/^/[rru]^{\subseteq} & \mathbb{G}^2_{\delta+1, \varepsilon} \ar[ru]_{\nexists\dot{\bold f}}   & \\
	\mathbb{G}^1_{\delta, \varepsilon} \ar [u]^\subseteq \ar[r]_\subseteq & \mathbb{G}^2_{\delta, \varepsilon}\ar[u] ^\subseteq \ar@/_/[ruu]_{\bold{f}^\varepsilon_\delta}  &
}
\]
\item[$(e4)$] $\mathbb{L}^\delta_\zeta \subseteq \dot{\mathbb{L}}_\zeta$ for $\zeta \leq \delta,$

\item[$(e5)$] $\mathbb{L}^\delta_\zeta = \mathbb{L}^\delta_{\zeta+1} \cap \dot{\mathbb{L}}_\zeta,$

\item[$(e6)$] $\dot{\mathbb L}_{\zeta+1}/ (\mathbb{L}^\delta_{\zeta+1}+\dot{\mathbb L}_\zeta)$ is free.
\end{enumerate}
		\end{enumerate}
	\end{enumerate}
	For $\alpha=0,$ set $\mathbb{G}^1_{0,\varepsilon} = \mathbb{G}^2_{0,\varepsilon} = \{0\}.$
	For limit ordinal $\delta,$ set $\mathbb{G}^\iota_{\delta,\varepsilon} = \bigcup\limits_{\alpha < \delta}\mathbb{G}^\iota_{\alpha,\varepsilon}$.
	Let us show that items (c) and (d) continue to hold. By the induction hypothesis and for all $\alpha < \beta < \delta$ with $\alpha \notin S$,
	$\mathbb{G}^\iota_{\beta, \varepsilon} / \mathbb{G}^\iota_{\alpha, \varepsilon}$ and
	$\mathbb{G}^2_\beta/(\mathbb{G}^1_\beta + \mathbb{G}^2_\alpha)$ are free, and since we have a club $C_\delta$ of $\delta$
	which is disjoint to $S,$ it immediately follows that
	the groups $\mathbb{G}^\iota_{\delta, \varepsilon}$, $\mathbb{G}^\iota_{\delta, \varepsilon}/ \mathbb{G}^\iota_{\alpha, \varepsilon}$ and  $\mathbb{G}^2_{\delta, \varepsilon}/(\mathbb{G}^1_{\delta, \varepsilon} + \mathbb{G}^2_{\alpha, \varepsilon})$ are
	free for all $\alpha < \delta$ with $\alpha  \notin S$.
	
Let $\iota=1, 2$.	Now suppose that $\delta < \lambda$ and we have defined the groups $\mathbb{G}^\iota_{\alpha, \varepsilon}$
 and ordinals $\alpha \leq \delta$.
	We would like to define the groups $\mathbb G^1_{\delta+1, \varepsilon}$
	and $\mathbb G^2_{\delta+1, \varepsilon}$ so that the (a)-(e) continue to hold.

	In the case $\delta \notin S_\varepsilon,$ we set
	$$\mathbb{G}^1_{\delta+1, \varepsilon}:=\mathbb{G}^1_{\delta, \varepsilon}$$
	and  	
	$$\mathbb{G}^2_{\delta+1, \varepsilon}:=(\mathbb{G}^1_{\delta+1, \varepsilon} \oplus_{\mathbb{G}^1_{\delta+1, \varepsilon}}\mathbb{G}^2_{\delta, \varepsilon}) \oplus \bigoplus\limits_{n<\omega} y_{\delta, n} \mathbb Z.$$
	It is not difficult to show that items (a)-(d) continue to hold and there is nothing to prove for case (e).
	
	Now suppose that $\delta \in S_\varepsilon$. We define the groups $\mathbb{G}^1_{\delta+1, \varepsilon}$
	and $\mathbb{G}^2_{\delta+1, \varepsilon}$  such that items (a)-(e) above continue to hold, and further we have:
	
	\begin{enumerate}
		
		\item[$(f)$] $\mathbb{G}^2_{\delta+1,\varepsilon}/(\mathbb{G}^1_{\delta,\varepsilon}
		+ \mathbb{G}^2_{\delta,\varepsilon})$ is not free,	
		
		\item[$(g)$] if $\gamma \in
		\delta \backslash S_\varepsilon$, then $
		\mathbb{G}^2_{\delta+1,\varepsilon}/(\mathbb{G}^1_{\delta,\varepsilon} +
		\mathbb{G}^2_{\gamma,\varepsilon})$ is free.
		
		
		
		
	\end{enumerate}
	As $\delta \in S_\varepsilon$, we have $\cf(\delta) = \aleph_0$, so
	let $\langle \gamma_{\delta,n}:n < \omega\rangle$ be an increasing
	sequence of successor ordinals $< \delta$ with limit $\delta$.
	
	For any $n<\omega, \gamma_{\delta, n} \notin S_\varepsilon,$  it is easily seen that
	the abelian groups $\mathbb G^1_{\gamma_{\delta, n}, \varepsilon} \subseteq \mathbb{G}^2_{\gamma_{\delta, n}, \varepsilon}$
	satisfy the hypotheses of Lemma \ref{nonext}, hence by the lemma, we can find the groups $\mathbb G^1_{\delta+1, \varepsilon} \subseteq \mathbb G^2_{\delta+1, \varepsilon}$ such that there for all constructions $\mathbb L \supseteq \mathbb G^1_{\delta+1, \varepsilon}$ as described in (e),
	there is no $\bold f \in \Hom(\mathbb{G}^2_{\delta+1, \varepsilon}, \mathbb L)$
	such that $\bold f \supseteq \bold  f^\varepsilon_\delta \cup \id_{\mathbb G^1_{\delta+1, \varepsilon}}$.

	This finishes our inductive construction. For $\iota=1, 2$
	and $\varepsilon < \lambda$, we  define:
	\[
	\mathbb{G}^\iota_\varepsilon = \bigcup\limits_{\alpha < \lambda} \mathbb{G}^{\iota}_{\alpha, \varepsilon}.
	\]
	Set also
	\begin{itemize}
		\item[(h1):] $\mathbb{G}_{\alpha, \varepsilon}=\mathbb{G}^{2}_{\alpha, \varepsilon}/ \mathbb{G}^{1}_{\alpha, \varepsilon}$,
		
		\item[(h2):] $\mathbb{G}_{\varepsilon}=\bigcup\limits_{\alpha < \lambda}\mathbb{G}_{\alpha, \varepsilon},$
		
		\item[(h3):] $\bar{\mathbb{G}}=\langle \mathbb{G}_{\varepsilon}: \varepsilon < \lambda \rangle.$
		
	\end{itemize}
	
	Let us show that the properties.

$(i)$: This is clear.
 	
$(ii)$: We apply the properties taken from items (c) and (d) of the construction, along with freeness  $\mathbb{G}^1_\varepsilon \subseteq \mathbb{G}^2_\varepsilon$
	to deduce  $ \mathbb{G}_{\varepsilon}$ is  strongly $\lambda$-free as witnessed by the sequence
	$$\mathcal{S}_\varepsilon=\{ \mathbb{G}_{\alpha, \varepsilon}: \alpha \in \lambda \setminus S_\varepsilon   \},$$for more details, see \cite[IV.1.11]{EM02}.

  $(iii)$: Suppose by contradiction that there are $\varepsilon < \lambda$, a construction
	$\bar{\mathbb{L}}=\langle \mathbb{L}_\varepsilon: \varepsilon \leq \lambda      \rangle$   by $\{\mathbb{G}_\alpha:\alpha \in \lambda \setminus \varepsilon\}$
	over $\mathbb{G}^1_\varepsilon$ and there is a  homomorphism ${\bf g}$ from
	$\mathbb{G}^2_\varepsilon$ into $\mathbb{L}_\lambda$ which extends $\id_{\mathbb{G}^1_\varepsilon}$:$$\xymatrix{
		&&  \mathbb{G}^1_\alpha \ar[r]^{\subseteq}\ar[d]_{\id}&\mathbb{G}^2_\alpha\ar[d]^{{\bf g}}\\
		&& \mathbb{G}^1_\alpha \ar[r]^{\subseteq}& \mathbb{L}
		&&&}$$
 Without loss of generality we can assume that $\mathbb{L}_\lambda$
	has size $\lambda$ and that its universe is $\lambda$. The
	set
	\begin{center}
		$E=\{\delta < \lambda: {\bf g} \restriction \delta: \delta \to \delta$ and $\bar{\mathbb{L}} \restriction \delta= \bar{\mathbb{L}}_\delta \restriction \delta$
 and $\mathbb{L}_\lambda \restriction \delta$ is a subgroup of $\mathbb{L} \}$
	\end{center}
	is a club, thus we can find some $\delta \in E \cap S_\varepsilon$ such that:
	\begin{itemize}
		\item[(i1):] ${\bf g} \restriction \delta=f^\varepsilon_\delta$,
		\item[(i2):] $\mathbb{L}_\kappa \restriction \delta=\mathbb{H}^{2, \varepsilon}_\delta$,
		\item[(i3):] $\mathbb{H}^{1, \varepsilon}_\delta=\mathbb{G}^1_{\delta, \varepsilon},$
		\item[(i4):] $\mathbb{G}^2_\varepsilon \cap \delta=\mathbb{G}^2_{\delta, \varepsilon}.$
	\end{itemize}
	Now note that $f={\bf g}\restriction \mathbb{G}^2_{\delta+1, \varepsilon}: \mathbb{G}^2_{\delta+1, \varepsilon} \to \mathbb{L}_\lambda$ is such that
	${f}^\varepsilon_\delta\cup \id_{\mathbb{G}_{\delta+1, \varepsilon}^1} \subseteq{f}$,
	which is in contradiction with clause (e) of the construction.

	$(2)$: The proof is similar to the proof of (1), this time, we find  the following family $$\langle S_\varepsilon: \varepsilon < 2^{\lambda} \rangle$$ of
 almost disjoint subsets of $\lambda$ such that $ \Diamond_{S_\varepsilon}$ holds for all $\varepsilon$. By \cite{devlin}, such a sequence exists.
\end{proof}

 Now, we are ready to prove Theorem (B):

\begin{theorem}
	\label{a1}
	Let $S \subseteq
	S^\lambda_{\aleph_0}$ be stationary non-reflecting and suppose
	$\diamondsuit_S$ holds. Suppose one of the followings:
	\begin{enumerate}
		\item[$(1)$]	 $\mu = \lambda = \cf(\lambda) > \aleph_0$, or
		
		\item[$(2)$]  $\mu = 2^\lambda,\lambda =  \cf(\lambda) > \aleph_0$.
	\end{enumerate} Then
	there are sequences $\langle \mathbb G_\alpha: \alpha < \mu \rangle$ and  $\langle \mathbb K_\alpha: \alpha < \mu \rangle$ of $\lambda$-free abelian groups such that for all $\alpha < \mu, |\mathbb G_\alpha|=\lambda,$ $|\mathbb K_\alpha|=2^\lambda$ and for all $\alpha, \beta < \mu,$
	$$\Ext(\mathbb G_\alpha, \mathbb K_\beta) = 0 \Leftrightarrow \alpha R \beta.$$
\end{theorem}

\begin{proof}
	This follows from Theorem \ref{a2} and Theorem  \ref{underdiamond}.
\end{proof}

\section{Representing a bipartite graph by Ext in ZFC}

In this section we show that it is possible to remove the diamond principle from the construction of Section 4 and get ZFC result. The main result is Theorem \ref{a1z}.   This  answers  Herden's question for the case of bipartite graphs.
We will do this  by using a simple version of  Shelah's black
box. In this case, the groups $\mathbb G_\alpha$ that we construct are not $\lambda$-free, but just $\aleph_1$-free.
\begin{notation}\label{hchi}
Let $\chi$ be   be infinite cardinal. By $\cH(\chi)$ we mean the collection of sets of hereditary cardinality
less than $\chi$.
\end{notation}
Let us start by stating the version of the black box	we are using in this paper.

\begin{theorem}	
	\label{blackboxthm} Let $\chi$, $\lambda$ and $\mu$   be infinite cardinals
	such that $\lambda=\mu^+, \mu^{\aleph_0}=\mu$, and $E_0, \ldots, E_{m-1}$ are pairwise
	disjoint stationary subsets of $\lambda$ consisting of ordinals of cofinality
	$\omega$, and $\chi>\lambda$. Let $N$ be an expansion in a countable language of
	$(\cH(\chi),\in, \lhd, \lambda)$ where  $\lhd$ is a well ordering of $\cH(\chi)$. Then there is a
	family of countable sets $\{(M_i, X_i): i \in I\}$ such that the following
	properties hold:
	
	\begin{enumerate}
		\item[(a)]
		$M_i\prec N$ and $X_i\subset \lambda$.
		\item[(b)]
		Let $\delta(i):=\sup(M_i\cap \lambda)$. If  $\delta(i)=\delta(j)$, then $(M_i, X_i)\cong(M_j, X_j)$
		and $M_i\cap M_j\cap \lambda$ is an initial segment	of $M_i\cap \lambda$.
		\item[(c)] For all $X\subset \lambda$, all $\ell<m$, the following  set $$\{\delta\in E_{\ell}:\exists i \emph{  such that }\delta(i) =\delta \emph{ and } (M_i, X_i)\equiv_{M_i\cap\lambda} (N, X)\}$$ is stationary in $\lambda$.
	\end{enumerate}	
\end{theorem}	

\begin{proof}
	See \cite[Page 444]{EM02}.
\end{proof}

\begin{notation}
	Given a torsion-free group $\mathbb{G}$ and a subgroup $\mathbb{H}\subseteq\mathbb{G}$. By $\mathbb{H}_{\ast}$
	we mean  the pure-closure, i.e.,  the smallest pure subgroup of $\mathbb{G}$ containing $\mathbb{H}$. In fact,
	$$\mathbb{H}_{\ast}=	\{g\in \mathbb{G}: ng\in \mathbb{H} \emph{ for some  nonzero }n\in\mathbb{Z} \}.$$
\end{notation}
The next task is to construct $\boxtimes_{\lambda,\mu}^1$ from Theorem \ref{a2} in ZFC:
\begin{theorem}
	\label{a52} Adopt one of the following assumptions:
	\begin{enumerate}
	\item[(1)] If $\lambda = \lambda^{\aleph_0},\mu = \lambda$,
	\item[(2)] If $\lambda = \lambda^{\aleph_0},\mu =
	2^\lambda$.
	\end{enumerate}	
Then $\boxtimes^1_{\lambda,\mu}$ from Lemma \ref{a2} holds.
\end{theorem}

\begin{proof}
(1):  By a result of Solovay (see \cite[Corollary II.4.9]{EM02}) we can find a partition
$\langle  S_\varepsilon: \varepsilon < \lambda     \rangle$  of $\lambda$ into $\lambda$ many disjoint stationary sets. Let  $$\mathbb{C}_\varepsilon:=\bigoplus\limits_{\nu < \lambda, n<\omega}(y_{\varepsilon, \nu, n}\mathbb Z \oplus z_{\varepsilon, \nu, n}\mathbb Z)$$ and
 recall that  its $\mathbb Z$-adic completion is denoted  by $\widehat{\mathbb C_\varepsilon}$. Clearly, $\widehat{\mathbb C_\varepsilon}$ has cardinality $\lambda,$
so we identify it with $\lambda$. Let also
\[
Y_\varepsilon: = \{ y_{\varepsilon, \nu, n}: \nu < \lambda, n<\omega                    \} \cup \{z_{\varepsilon, \nu, n}: \nu < \lambda, n<\omega  \}.
\]
Set $\chi= \lambda^{+3}$. The initial structure for the Black box, corresponding to the stationary set $S_\varepsilon$ is  as follows:
\[
N_\varepsilon:= (\cH(\chi), \in, \lhd, \lambda, \widehat{C_\varepsilon},  Y_\varepsilon),
\]
where $\widehat{C_\varepsilon}$ denotes the 3-ary relation on $\lambda$
which is the graph of the addition operation on the group $\widehat{\mathbb C_\varepsilon}$.
We take the first bijection $\bold g_\varepsilon: \lambda \times \lambda \to \lambda$ with respect to $\lhd$, and use it to identify
each $X \subseteq \lambda$ with a subset of $\widehat{C_\varepsilon} \times \widehat{C_\varepsilon}.$
Let $\{(M^\varepsilon_i, X^\varepsilon_i): i \in I_\varepsilon       \}$
be as in the statement of Theorem \ref{blackboxthm} when $m=1$ and $E_0=S_\varepsilon$, and note that for each $i\in I_\varepsilon$ and $\varepsilon<\lambda,$ $\bold g_\varepsilon \in M^\varepsilon_i$.

Let   $\iota=1, 2$. We proceed as in the previous section and
for a given $\varepsilon< \lambda$, by induction on $\alpha < \lambda$, we choose the sequences
$\bar{\mathbb{G}}^\iota_\varepsilon = \langle
\mathbb{G}^\iota_{\alpha,\varepsilon}:\alpha < \lambda\rangle$ equipped with the following five items:

\begin{enumerate}
	\item[$(a)$] $\bar{\mathbb{G}}^\iota_\varepsilon$ is an increasing
	and continuous sequence of abelian groups $\mathbb{G}^\iota_{\alpha,\varepsilon}=(G^\iota_{\alpha, \varepsilon}, +, 0)$.
	
	\item[$(b)$]   $\mathbb{G}^\iota_{0,\varepsilon} = \{0\}$ and for $\alpha < \lambda,$
	the universe of $\mathbb{G}^2_{\alpha,\varepsilon}$, namely $G^2_{\alpha, \varepsilon}$ is an ordinal
	$\gamma_{\alpha,\varepsilon} < \lambda$.
	
	\item[$(c)$]  The following three properties hold:
	\begin{enumerate}
		\item[$(c1)$:]  $\mathbb{G}^1_{\alpha,\varepsilon}
		\subseteq \mathbb{G}^2_{\alpha,\varepsilon}$ are $\aleph_1$-free,
		\item[$(c2) $:] if $\alpha < \beta$ then $
		\mathbb{G}^2_{\alpha, \varepsilon} \cap \mathbb{G}^1_{\beta, \varepsilon} = \mathbb{G}^1_{\alpha, \varepsilon}$,
		\item[$(c3)$:]  if $\alpha < \beta$ and $\alpha \notin
		S_\varepsilon$, then  $\mathbb{G}^2_\beta/(\mathbb{G}^1_\beta + \mathbb{G}^2_\alpha)$ is $\aleph_1$-free.
	\end{enumerate}
	\item[$(d)$] If $\alpha < \beta$ and $\alpha \notin S_\varepsilon$,
	then $\mathbb{G}^1_{\beta,\varepsilon}/\mathbb{G}^1_{\alpha,\varepsilon}$
	and $\mathbb{G}^2_{\beta,\varepsilon}/\mathbb{G}^2_{\alpha,\varepsilon}$ are
	$\aleph_1$-free.
	
	\item[$(e)$] If $\delta \in S_\varepsilon,$  then there are no $f$ and $ \langle  \dot{\mathbb{L}}_\zeta: \zeta \leq \delta \rangle$ equipped with the following properties:
\begin{enumerate}
\item[$(e1)$] $\langle  \dot{\mathbb{L}}_\zeta: \zeta \leq \delta \rangle$ is a construction over $\mathbb{G}_{\delta+1, \varepsilon}^1,$

\item[$(e2)$]  $f \in \Hom(\mathbb{G}^2_{\delta+1, \varepsilon}, \dot{\mathbb{L}}_{\delta})$,

\item[$(e3)$] $f$ extends $f^\varepsilon_{\delta, i} \cup \id_{\mathbb{G}^1_{\delta+1, \varepsilon}}$
where $i \in I_\varepsilon$	and

$$\langle
f^\varepsilon_{\delta, i}, \bar{\mathbb{L}}^\varepsilon_{\delta, i}=\langle \mathbb{L}^\varepsilon_{\delta, i, \zeta}: \zeta \leq \delta \rangle \rangle$$ is coded by $X^\varepsilon_i$, under the identification given by $\bold g_\varepsilon$. Also, $f^\varepsilon_{\delta, i}$ is in $\Hom(\mathbb{G}^2_{\delta, \varepsilon}, \mathbb{L}^\varepsilon_{\delta, i, \delta})$ and $\bar{\mathbb{L}}^\varepsilon_{\delta, i}$ is a construction:

\[
\xymatrix@=3pc{
	& \dot{\mathbb{L}}_{\delta} \\
	\mathbb{G}^1_{\delta+1, \varepsilon} \ar[r]_\subseteq\ar [ru]^{\subseteq} & \mathbb{G}^2_{\delta+1, \varepsilon} \ar[u]^{\nexists{f}}   & \mathbb{G}^2_{\delta, \varepsilon}\ar[lu]_{{f}^\varepsilon_{\delta,i}} \ar[l]^{\subseteq} &
}
\]


\item[$(e4)$] $\mathbb{L}^\delta_\zeta \subseteq \dot{\mathbb{L}}_\zeta$ for $\zeta \leq \delta,$

\item[$(e5)$] $\mathbb{L}^\delta_\zeta = \mathbb{L}^\delta_{\zeta+1} \cap \dot{\mathbb{L}}_\zeta,$

\item[$(e6)$] $\dot{\mathbb L}_{\zeta+1}/ (\mathbb{L}^\delta_{\zeta+1}+\dot{\mathbb L}_\zeta)$ is free.
\end{enumerate}
\end{enumerate}
For $\alpha=0,$ set $\mathbb{G}^1_{0,\varepsilon} = \mathbb{G}^2_{0,\varepsilon} = \{0\}.$
For the limit ordinal $\delta,$ we set $\mathbb{G}^\iota_{\delta,\varepsilon}: = \bigcup\limits_{\alpha < \delta}\mathbb{G}^\iota_{\alpha,\varepsilon}$.
Let us show that items (c) and (d) continue to hold. By the induction hypothesis and for all $\alpha < \beta < \delta$ with $\alpha \notin S$,
$\mathbb{G}^\iota_{\beta, \varepsilon} / \mathbb{G}^\iota_{\alpha, \varepsilon}$ and
$\mathbb{G}^2_\beta/(\mathbb{G}^1_\beta + \mathbb{G}^2_\alpha)$ are $\aleph_1$-free, and since we have a club $C_\delta$ of $\delta$
which is disjoint to $S,$ it immediately follows from Lemma \ref{aleffree} that the groups $\mathbb{G}^\iota_{\delta, \varepsilon}$, $\mathbb{G}^\iota_{\delta, \varepsilon}/ \mathbb{G}^\iota_{\alpha, \varepsilon}$ and  $\mathbb{G}^2_{\delta, \varepsilon}/(\mathbb{G}^1_{\delta, \varepsilon} + \mathbb{G}^2_{\alpha, \varepsilon})$ are $\aleph_1$-free for all $\alpha < \delta$ with $\alpha  \notin S_\varepsilon$.

Now suppose that $\delta < \lambda$ and we have defined the groups $\mathbb{G}^\iota_{\alpha, \varepsilon}$
for $\iota=1, 2$ and ordinals $\alpha \leq \delta$.
We would like to define the groups $\mathbb G^1_{\delta+1, \varepsilon}$
and $\mathbb G^2_{\delta+1, \varepsilon}$ so that the (a)-(e) continue to hold.

If $\delta \notin S_\varepsilon,$ then we set
$$\mathbb{G}^1_{\delta+1, \varepsilon}:=\mathbb{G}^1_{\delta, \varepsilon}$$
and  	
$$\mathbb{G}^2_{\delta+1, \varepsilon}:=(\mathbb{G}^1_{\delta+1, \varepsilon}\oplus_{\mathbb{G}^1_{\delta+1, \varepsilon}}\mathbb{G}^2_{\delta, \varepsilon}) \oplus \bigoplus\limits_{n<\omega} y_{\varepsilon, \delta, n} \mathbb Z.$$
We leave to the reader to check that the items presented from  (a) to (d) all are valid, and recall that there is nothing to prove for case (e).

Now suppose that $\delta \in S_\varepsilon$. We define the groups $\mathbb{G}^1_{\delta+1, \varepsilon}$
and $\mathbb{G}^2_{\delta+1, \varepsilon}$  such that items (a)-(e) above continue to hold, and further we have:

\begin{enumerate}
	
	\item[$(f)$] $\mathbb{G}^2_{\delta+1,\varepsilon}/(\mathbb{G}^1_{\delta,\varepsilon}
	+ \mathbb{G}^2_{\delta,\varepsilon})$ is not free,	
	
	\item[$(g)$] if $\gamma \in
	\delta \backslash S_\varepsilon$, then $
	\mathbb{G}^2_{\delta+1,\varepsilon}/(\mathbb{G}^1_{\delta,\varepsilon} +
	\mathbb{G}^2_{\gamma,\varepsilon})$ is $\aleph_1$-free.
	
	
	
	
\end{enumerate}
As $\delta \in S_\varepsilon$, we have $\cf(\delta) = \aleph_0$, so
let $\langle \gamma_{\delta,n}:n < \omega\rangle$ be an increasing
sequence of successor ordinals $< \delta$ with limit $\delta$.

Let $n<\omega$ be such that $ \gamma_{\delta, n} \notin S_\varepsilon$. It turns out that
the abelian groups $\mathbb G^1_{\gamma_{\delta, n}, \varepsilon} \subseteq \mathbb{G}^2_{\gamma_{\delta, n}, \varepsilon}$
are suited well in  the hypotheses of Lemma \ref{nonext}.

The notation $\Sigma_{\delta, \varepsilon}$ stands for the following set:
$$\left\{i \in I_\varepsilon: \delta(i)=\delta\emph{ and }X^\varepsilon_i\emph{
codes }\langle
f^\varepsilon_{\delta, i}, \bar{\mathbb{L}}^\varepsilon_{\delta, i}=\langle \mathbb{L}^\varepsilon_{\delta, i, \zeta}: \zeta \leq \delta \rangle \rangle\emph{
as in }(e3)\right\}.$$

Given any $i \in \Sigma_{\delta, \varepsilon}$, and according to  Lemma \ref{nonext}, we can find the $\aleph_1$-free groups $\mathbb G^{1, i}_{\delta+1, \varepsilon} \subseteq \mathbb G^{2, i}_{\delta+1, \varepsilon}$ such that there for all constructions $\langle\dot{\mathbb L}_\zeta: \zeta \leq \delta \rangle$ over $\mathbb G^{1,i}_{\delta+1, \varepsilon}$
as in item (e), if we set $\mathbb{L}=\dot{\mathbb{L}}_\delta$, then
there is no $\bold f \in \Hom(\mathbb{G}^{2, i}_{\delta+1, \varepsilon}, \mathbb L)$
such that $\bold f \supseteq f^\varepsilon_{\delta, i} \cup \id_{\mathbb G^{1, i}_{\delta+1, \varepsilon}}$.
For $\iota=1,2$ we look at \[
\mathbb G^\iota_{\delta+1, \varepsilon} := \langle  \mathbb G^\iota_{\delta, \varepsilon} \cup \bigcup\{\mathbb G^{\iota, i}_{\delta+1, \varepsilon}: i \in \Sigma_{\delta, \varepsilon }   \}                             \rangle_{\ast},
\] i.e.,
 the pure closure of  $\langle  \mathbb G^\iota_{\delta, \varepsilon} \cup \bigcup\{\mathbb G^{\iota, i}_{\delta+1, \varepsilon}: i \in \Sigma_{\delta, \varepsilon} \}  \rangle$
in $\widehat{\mathbb C_\varepsilon}$.
Let us show that the hypothesis (a)-(e) hold. We first show that the group $\frac{\mathbb{G}^2_{\delta+1, \varepsilon}}{\mathbb{G}^2_{\alpha, \varepsilon}}$ is  $\aleph_1$-free, provided that $\alpha$ is
	a successor ordinal.
	 Let $\mathbb{K}$
	be any countable subgroup of $\mathbb{G}^2_{\delta+1, \varepsilon}$. We are going to show that $\frac{\mathbb{K}+\mathbb{G}^2_{\delta+1, \varepsilon}}{\mathbb{G}^2_{\alpha, \varepsilon}}$ is  free.
There is an $\omega$-sequence $\{i_m:m \in \omega\}$ together with  a countable subgroup $\mathbb{I}\subset\mathbb{G}^2_{\delta, \varepsilon}$   such that $\mathbb{K}$ is the subgroup generated
	by $\mathbb{I}$   together with some countable subset $\{w_{n,{i_m}}:n, m \in \omega\}$ of $\bigcup\limits_{i\in \Sigma_{\delta,\varepsilon}}\mathbb G^{\iota, i}_{\delta+1, \varepsilon}$. We can assume
	that for all $n, m$, $y_{\varepsilon, \alpha_{n,{i_m}}}\in \mathbb{I}$. Choose an increasing sequence of
	ordinals $\{\alpha_k: k < \omega\}$ with limit $\delta$ such that $\alpha_0=\alpha$ and for all $m\in\omega$
	and all but finitely many of $n$, $\alpha_{n,i_m}\in\{\alpha_k :k < \omega \}$. Notice that
	for any successor ordinal $\gamma < \delta$
 $$\mathbb{I}/ (\mathbb{I} \cap \mathbb{G}^2_{\gamma, \varepsilon} ) \cong (\mathbb{I} + \mathbb{G}^2_{\gamma, \varepsilon} )/  \mathbb{G}^2_{\gamma, \varepsilon}$$
 which is
	free by the induction hypothesis. So for such $\gamma$, $\mathbb{I}   \cap  \mathbb{G}^2_{\gamma, \varepsilon} $ is a direct summand of $\mathbb{I}$.
	Inductively choose
	subgroups $\mathbb{I}_k$ so that:$$\mathbb{I}\cap   \mathbb{G}^2_{{\alpha_k+1}, \varepsilon}   \oplus \mathbb{I}_k=\mathbb{I}\cap   \mathbb{G}^2_{{\alpha_{k+1}}, \varepsilon}         $$ for all $k$.   Hence
$$\mathbb{I}=\bigoplus_k \mathbb{I}_k\oplus\bigoplus\mathbb{Z}y_{\varepsilon, \alpha_k,n}.$$
	
In view of Theorem
\ref{blackboxthm}(b), we are able to choose $\{n(m):m < \omega\}$ so that:
\begin{itemize}
\item the collections $\{\alpha_{n,i_m} :n(m) <  n\}$ are pairwise disjoint.	

\item $\{\alpha_{n,i_m} :n(m) <  n\}\subset  \{\alpha_k: k < \omega\}$.
\end{itemize}
We observe that $\frac{\mathbb{K}+\mathbb{G}^2_{\delta, \varepsilon}}{\mathbb{G}^2_{\delta, \varepsilon}}$
	is isomorphic to the direct sum of  $\bigoplus_k \mathbb{I}_k$ together with the group
	freely generated by $$\{w_{n,{i_m}}:  n(m)\leq n \emph{ and }m \in\omega \} $$ and $$\{y_{\varepsilon, \alpha_k,n}:\forall
	m < \omega \text{~ and~ }n(m) \leq n, \alpha_k\neq \alpha_{n,i_m}\}.$$
From this, the claim  follows. By  a similar argument,
the group $\frac{\mathbb{G}^1_{\delta+1, \varepsilon}}{\mathbb{G}^1_{\alpha, \varepsilon}}$ is  $\aleph_1$-free, provided that $\alpha$ is
	a successor ordinal.

 In the same vein, we also observe that the group  $\mathbb G^1_{\delta+1, \varepsilon} $
is $\aleph_1$-free and the hypotheses (a)-(e) continue to hold.

The rest of the argument is similar to Theorem  \ref{underdiamond}. Let us elaborate the main idea of the proof.
 For $\iota=1, 2$
	and $\varepsilon < \lambda$, we  define:
	\[
	\mathbb{G}^\iota_\varepsilon = \bigcup\limits_{\alpha < \lambda} \mathbb{G}^{\iota}_{\alpha, \varepsilon}.
	\]
	Set also
	\begin{itemize}
		\item[(h1):] $\mathbb{G}_{\alpha, \varepsilon}=\mathbb{G}^{2}_{\alpha, \varepsilon}/ \mathbb{G}^{1}_{\alpha, \varepsilon}$,
		
		\item[(h2):] $\mathbb{G}_{\varepsilon}=\bigcup\limits_{\alpha < \lambda}\mathbb{G}_{\alpha, \varepsilon},$
		
		\item[(h3):] $\bar{\mathbb{G}}=\langle \mathbb{G}_{\varepsilon}: \varepsilon < \lambda \rangle.$
		
	\end{itemize}
	
	Let us show that $\boxtimes^1_{\lambda,\mu_1}$ is satisfied. By items (c) and (d) of the construction, $\mathbb{G}^1_\varepsilon \subseteq \mathbb{G}^2_\varepsilon$
	are $\aleph_1$-free
	and $ \mathbb{G}_{\varepsilon}$ is   $\aleph_1$-free as well.
	Let us now show that $\boxtimes^1_{\lambda,\mu_1}(c)$ is satisfied as well. Suppose by contradiction that there are $\varepsilon < \lambda$, a construction
	$\mathbb{L}$   by $\{\mathbb{G}_\alpha:\alpha \in \lambda \setminus \varepsilon\}$
	over $\mathbb{G}^1_\varepsilon$, witnessed by $\langle \mathbb{L}_\varepsilon: \varepsilon \leq \lambda \rangle,$ and there is a  homomorphism ${\bf g}$ from
	$\mathbb{G}^2_\varepsilon$ into $\mathbb{L}$ which extends $\id_{\mathbb{G}^1_\varepsilon}$:
	$$\xymatrix{
		&&  \mathbb{G}^1_\alpha \ar[r]^{\subseteq}\ar[d]_{=}&\mathbb{G}^2_\alpha\ar[d]^{{\bf g}}\\
		&& \mathbb{G}^1_\alpha \ar[r]^{\subseteq}& \mathbb{L}
		&&&}$$
Without loss of generality we can assume that $\mathbb{L}$
	has size $\lambda$ and that its universe is $\lambda$.
Let $X \subseteq \lambda$ be codes $\langle \bf g, \langle \mathbb{L}_\varepsilon: \varepsilon \leq \lambda \rangle \rangle$.
The
	set
	\begin{center}
		$E:=\{\delta < \lambda: {\bf g} \restriction \delta: \delta \to \delta$, $X \cap \delta$ codes $\langle {\bf g}\restriction \delta, \langle \mathbb{L}_\zeta: \zeta \leq \delta \rangle          \rangle$  and $\mathbb{L} \restriction \delta\subseteq_{group} \mathbb{L} \}$
	\end{center}
	is a club, thus we can find some
$\delta \in E \cap S_\varepsilon$ and some $i$ with $\delta(i)=i$ such that $(M^\varepsilon_i, X^\varepsilon_i) \equiv_{M^\varepsilon_i \cap \lambda} (N_\varepsilon, X)$. It then follows that
$M^\varepsilon_i \cap \mathbb{G}^2_{\varepsilon}= M^\varepsilon_i \cap \mathbb{G}^{2, i}_{\delta, \varepsilon}$, and since $(M^\varepsilon_i, X^\varepsilon_i) \equiv_{M^\varepsilon_i \cap \lambda} (N_\varepsilon, X)$,
we can easily observe that
$X^\varepsilon_i$ codes $$\langle {\bf g} \restriction M^\varepsilon_i \cap \mathbb{G}^{2, i}_{\delta, \varepsilon},      \langle \mathbb{L}_\zeta \cap M^\varepsilon_i: \zeta \leq \lambda \rangle      \rangle.$$ Now by elementarily and the choice of $\bf g$,
\begin{center}
$(M_i^\varepsilon, X^\varepsilon_i)\models$``$X^\varepsilon_i$ codes a homomorphism $f^\varepsilon_{\delta, i}$ from $\mathbb{G}^{2, i}_{\delta, \varepsilon}$.
\end{center}
It follows that $i \in \Sigma_{\delta, \varepsilon},$ and hence by our construction,
there is no $\bold f \in \Hom(\mathbb{G}^{2, i}_{\delta+1, \varepsilon}, \mathbb L)$
such that $\bold f$ extends $f^\varepsilon_{\delta, i} \cup \id_{\mathbb G^{1, i}_{\delta+1, \varepsilon}}$. This property is conveniently summarized by the subjoined  diagram:

	
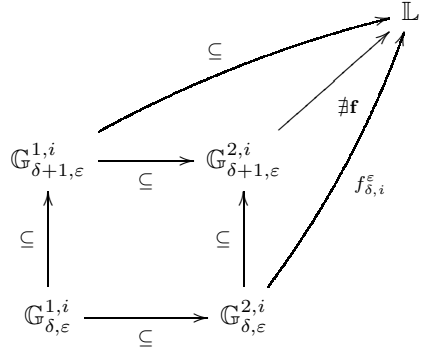
\[
\xymatrix@=3pc{
	& &  \mathbb L\\
	\mathbb G^{1, i}_{\delta+1, \varepsilon} \ar[r]_{\subseteq}\ar@/^/[rru]^{\subseteq} & \mathbb{G}^{2, i}_{\delta+1, \varepsilon} \ar[ru]_{\nexists{\bold f}}   & \\
	\mathbb{G}^{1, i}_{\delta, \varepsilon} \ar [u]^\subseteq \ar[r]_\subseteq & \mathbb{G}^{2, i}_{\delta, \varepsilon}\ar[u] ^\subseteq \ar@/_/[ruu]_{f^\varepsilon_{\delta, i}}  &
}
\]
This is not possible, as $\bf g$ is such an extension. This shows that
there is no  homomorphism ${\bf g}$ as above and the result follows.
	
$(2)$: This is similar to the proof of (1). Take a sequence $\langle  S_\varepsilon: \varepsilon < 2^\lambda    \rangle$
of almost disjoint subsets of $\lambda$ such that each $S_\varepsilon$ is stationary and proceed as before.
\end{proof}

Now, we are ready to prove Theorem (C) from \S1:

\begin{theorem}
	\label{a1z}
Let $\lambda = \lambda^{\aleph_0},\mu = 2^\lambda$ and let $R \subseteq \mu
\times \mu$  be a relation. Then
there are families
$\langle  \mathbb G_\alpha: \alpha < \mu \rangle,$ and  $\langle \mathbb K_\alpha: \alpha <  \mu  \rangle$ of $\aleph_1$-free abelian groups
such that:
\begin{enumerate}
	\item  for all $\alpha < \mu, \mathbb G_\alpha$ has size $\lambda$
	and $\mathbb K_\alpha$ has size $2^\lambda,$
	\item for all $\alpha, \beta < \mu,$
	$$\Ext(\mathbb G_\alpha, \mathbb K_\beta) = 0 \iff \alpha R \beta.$$
\end{enumerate}
\end{theorem}

\begin{proof}
	This follows using Theorem \ref{a2} and Theorem \ref{a52}.
\end{proof}	
Let us close this section by showing the relation between
Theorem \ref{a1z} and the result
of G\"obel,
Shelah and Wallutis \cite{GSW} stated in the introduction.
\begin{lemma}
\label{our1}
Let $\mu$ be an infinite cardinal. Then (1) implies (2), where:
\begin{enumerate}
\item  There are  abelian groups
$\langle \mathbb G_X ,\mathbb H^X: X \subseteq \mu \rangle$ such that
  for all $X, Y \subseteq \mu,$
 $$\Ext(\mathbb G_Y, \mathbb H^X) = 0 \iff Y \subseteq X.$$

\item If $(\mu, R)$ is a bipartite graph, then there are families
$\langle  \mathbb G_\alpha: \alpha < \mu \rangle,$ and  $\langle \mathbb K_\alpha: \alpha <  \mu  \rangle$ of  abelian groups
such that
for all $\alpha, \beta < \mu,$
	$$\Ext(\mathbb G_\alpha, \mathbb K_\beta) = 0 \iff \alpha R \beta.$$
\end{enumerate}
\end{lemma}
\begin{proof}
Suppose (1) holds as witnessed by the sequence  $\langle \mathbb G_X ,\mathbb H^X: X \subseteq \mu \rangle$ and let $(\mu, R)$ be a bipartite graph. Given any $\alpha, \beta < \mu$
set $Y_{\alpha}=\{\alpha\}$ and
$X_\beta=\{ \alpha < \mu: \alpha R \beta         \}$.
Now define

\begin{enumerate}
	\item[$i)$]  $\mathbb G_\alpha:=\mathbb G_{Y_\alpha}$ and
	\item[$ii)$]   $\mathbb K_\alpha:= \mathbb H^{X_\alpha}$.
\end{enumerate}
This is now straightforward to see
\[
\Ext(\mathbb G_\alpha, \mathbb K_\beta) =0 \iff Y_\alpha \subseteq X_\beta \iff  \alpha R \beta.
\]
Thus, the family $\langle  \mathbb G_\alpha, \mathbb K_\alpha: \alpha < \mu \rangle$
is as required.
\end{proof}
The converse of the above lemma also holds in the following sense:
\begin{lemma}
\label{our2}
Let $\mu$ be an infinite cardinal and set $\lambda=2^\mu$. Then (1) implies (2), where:
\begin{enumerate}
\item If  $(\lambda, R)$ is a bipartite graph, then there are families
$\langle  \mathbb G_\alpha: \alpha < \lambda \rangle,$ and  $\langle \mathbb K_\alpha: \alpha <  \lambda  \rangle$ of  abelian groups
such that
for all $\alpha, \beta < \lambda,$
	$$\Ext(\mathbb G_\alpha, \mathbb K_\beta) = 0 \iff \alpha R \beta.$$

\item  There are  abelian groups
$\langle \mathbb G_X ,\mathbb H^X: X \subseteq \mu \rangle$ such that
  for all $X, Y \subseteq \mu,$
 $$\Ext(\mathbb G_Y, \mathbb H^X) = 0 \iff Y \subseteq X.$$
\end{enumerate}
\end{lemma}
\begin{proof}
Suppose (1) holds. Let $\langle  X_\alpha: \alpha < \lambda        \rangle$
and $\langle Y_\beta: \beta < \lambda   \rangle$
be two enumerations of $\mathcal{P}(\mu)$. Let us define the bipartite graph $(\lambda, R)$
as
\[
\alpha R \beta \iff Y_\alpha \subseteq X_\beta.
\]
By (1), there exists a family $\langle  \mathbb G_\alpha, \mathbb K_\alpha: \alpha < \lambda \rangle,$  of  abelian groups
such that
for all $\alpha, \beta < \lambda,$
	$$\Ext(\mathbb G_\alpha, \mathbb K_\beta) = 0 \iff \alpha R \beta.$$
For $X, Y \subseteq \mu$, let $\alpha, \beta < \lambda$ be such that $X_\beta=X$ and $Y_\alpha=Y$ and set

\begin{enumerate}
	\item[$\bullet$]  $\mathbb G_{Y}:=\mathbb G_\alpha$;
	\item[$\bullet$]   $\mathbb H^X := \mathbb K_\beta$.
\end{enumerate}

Then
\[
\Ext(\mathbb G_Y, \mathbb H^X) = 0 \iff \Ext(\mathbb G_\alpha, \mathbb K_\beta) = 0 \iff \alpha R \beta \iff Y \subseteq X.
\]
The lemma follows.
\end{proof}

\section{Representing  a general graph by Ext}

In this section we consider general graphs and discuss if they can be represented by Ext as before. We do not know the result in ZFC, but we show the following consistency result which shows that there are no restrictions on such graphs in ZFC, as promised by Theorem (D) from the introduction:

\begin{theorem}
	\label{arbitrarygraphZFC}
	Suppose GCH holds and the pair $(S, R)$ is a graph  where $R \subseteq S \times S$, and  let $\lambda>|S|$ be an uncountable regular cardinal. Then there exists a cardinal preserving generic extension
	of the universe, and there is a family $\{\mathbb{G}_s\}_{s \in S}$ of $\lambda$-free abelian groups  such that
	\[
	\Ext(\mathbb{G}_s, \mathbb{G}_t)=0 \iff s R t.
	\]
\end{theorem}
The rest of this section is devoted to the proof of the above theorem. Before we go into the details, let us sketch the idea of the proof:
\begin{discussion}\label{dis}
	We first define a forcing notion $\bbP_{*}$ which adds a sequence $\langle \mathbb{G}_s: s \in S  \rangle$
	of  $\lambda$-free abelian groups of size $\lambda$ such that for each $s, t \in S$ if $(s, t) \notin R,$ then for some abelian group $\mathbb{H}_{s, t}$
	of size $\lambda$ there exists an exact sequence
	$$0\longrightarrow \mathbb{G}_s \longrightarrow \mathbb{H}_{s, t} \longrightarrow \mathbb{G}_t\longrightarrow 0$$
	which does not split. We apply this along with Fact \ref{yoneda} to deduce that
	$\Ext(\mathbb{G}_s, \mathbb{G}_t) \neq 0.$
	Then working in the generic extension by $\bbP_*,$ we define a cardinal preserving $\lambda$-support iteration forcing notion of length $\lambda^+$, which makes
	$\Ext(\mathbb{G}_s, \mathbb{G}_t)=0$ for all $s, t\in S$ with $s R t.$ This is done by adding a splitter for any
	exact sequence
	$$0\longrightarrow \mathbb{G}_s \longrightarrow \mathbb{H} \longrightarrow \mathbb{G}_t\longrightarrow 0,$$
	where $\mathbb{H}$ is an abelian group of size $\lambda.$ By using a suitable book-keeping argument, we make sure that at the end
	all such exact sequences are considered for all pairs $(s, t)  \in R.$ We also show that the exact sequence
	$$0\longrightarrow \mathbb{G}_s \longrightarrow \mathbb{H}_{s, t} \longrightarrow \mathbb{G}_t\longrightarrow 0$$
	still fails to split after the iteration, which will complete the proof.
\end{discussion}

Let us now go into the details of the proof.
Recall from Notation \ref{hchi} that $\cH(\chi)$ is the collection of sets of hereditary cardinality
less than $\chi$.

\begin{notation}
	Let $\Phi: \lambda^+ \to \cH(\lambda^+)$ be such that  $\Phi^{-1}[x] \subseteq \lambda^+$
	is unbounded for all $x \in \cH(\lambda^+)$.
\end{notation}

The  existence of $\Phi $ follows by the GCH assumption. We will use $\Phi$ as our book-keeping function.
Here, we define the forcing notion $\bbP_*.$
\begin{definition}  Let $\lambda = \text{ cf}(\lambda) > \aleph_0$
	be an uncountable regular cardinal, and  let $S$  be a set of cardinality $<\lambda$.
	Finally, let	$R \subseteq S \times S$.
	\label{pstarforcing}
	\begin{itemize}
		\item[(a)] The forcing notion $\bbP_*$ consists of conditions
		\[
		p=\langle  \langle\mathbb{G}^p_{s, \beta}: s \in S, \beta \leq \alpha_p \rangle,  E_p, \langle \mathbf x^p_{s,t,\beta}: (s, t) \notin R, \beta \leq \alpha_p             \rangle\rangle,
		\]
		where
		\begin{enumerate}
			\item $\alpha_p < \lambda$ is an ordinal,
			
			\item for each $s \in S,$ $\langle \mathbb{G}^p_{s,\beta}: \beta \leq \alpha_p\rangle$ is an increasing
			and continuous
			sequence of free Abelian groups from $\cH(\lambda)$,
			
			\item $E_p \subseteq (\alpha_p +1) \cap
			S^\lambda_{\aleph_0}$ does not reflect,
			
			\item  $ \mathbb{G}^p_{s, \beta}/\mathbb{G}^p_{s, \gamma}$ is free when $\gamma
			< \beta < \alpha_p,\gamma \notin E_p,$
			
			\item  if $(s,t) \notin R$,
			then
			$$
			\begin{CD}
			\mathbf x^p_{s,t,\beta} :=0@>>> \mathbb{G}^p_{t,\beta} @>f^p_{s,t, \beta}>>\mathbb{H}^p_{s,t, \beta} @>g^p_{s, t, \beta}>> \mathbb{G}^p_{s,\beta} @>>> 0
			\end{CD}
			$$
			is an exact sequence, all
			increasing with $\beta \le \alpha_p$, from $\cH(\lambda)$:
			$$
			\begin{CD}
			\mathbf x^p_{s, t, \alpha_p}:= 0@>>> \mathbb{G}^p_{t,\alpha_p} @>f^p_{s,t, \alpha_p}>>\mathbb{H}^p_{s,t, \alpha_p} @>g^p_{s, t, \alpha_p}>> \mathbb{G}^p_{s,\alpha_p} @>>> 0\\
			@. \subseteq @AAA \subseteq@AAA \subseteq @AAA   \\
			@.\vdots@.\vdots@.\vdots\\
			@.\subseteq @AAA\subseteq@AAA \subseteq @AAA   \\
			\mathbf x^p_{s, t, 1}:=0@>>> \mathbb{G}^p_{t,1} @>f^p_{s,t, 1}>>\mathbb{H}^p_{s,t, 1} @>g^p_{s, t, 1}>> \mathbb{G}^p_{s,1} @>>> 0\\
			@.\subseteq@AAA\subseteq @AAA \subseteq @AAA   \\
			\mathbf x^p_{s, t, 0}:=0@>>> \mathbb{G}^p_{t,0} @>f^p_{s,t,0 }>>\mathbb{H}^p_{s,t, 0} @>g^p_{s, t, 0}>> \mathbb{G}^p_{s,0} @>>> 0.\\
			\end{CD}
			$$
		\end{enumerate}
		\item[(b)] Given $p, q \in \bbP_*,$ let $p \leq q$ ($q$ is stronger than $p$) when
		\begin{enumerate}
			\item $\alpha_p \leq \alpha_q,$
			
			\item for all $s \in S$ and $\beta \leq \alpha_p,$ $\mathbb{G}^q_{s, \beta}=\mathbb{G}^p_{s, \beta}$,
			
			\item $E^q \cap (\alpha_p+1)=E^p$,
			
			\item If $(s, t) \notin R$ and $\beta \leq \alpha_p$, then $\mathbf x^q_{s,t,\beta}=\mathbf x^p_{s,t,\beta}.$
		\end{enumerate}
	\end{itemize}
\end{definition}

For simplicity of the reader we recall:

\begin{definition}Let $\bbP$ be a forcing notion and $\kappa$ be a regular cardinal.	\begin{enumerate}\item[(a)]  $\bbP$ is called
		$\kappa$-distributive if the intersection of less than $\kappa$-many dense open subsets of $\bbP$ is dense, or equivalently, forcing with
		$\bbP$ does not add any new sequences of ordinals of length less than $\kappa.$
		\item[(b)] $\bbP$ is called $\kappa$-strategically closed if for every $\alpha < \kappa$,  player II has a winning strategy in the following game:
		\begin{itemize}
			\item[$\Game_{\bbP}(\alpha)$]: the game has length $\alpha$ in which the players I and II take turns to play conditions from $\bbP$ for $\alpha$-many moves, with player I playing at odd stages and player II at even stages (including all limit stages). II must play $1_{\bbP}$ at move zero. Let $p_\beta$ b the condition played at move $\beta$; the player who played $p_\beta$ loses immediately unless $p_\beta \geq p_\gamma$ for all $\gamma < \beta.$
			If neither player loses at any stage $\beta < \alpha,$ then player II wins.
		\end{itemize}
		
		\item[(c)]$\bbP$ is called $\kappa$-c.c. if
		any antichain $A \subseteq \bbP$ has size less than $\kappa.$	\end{enumerate}
\end{definition}
It is evident that if $\bbP$ is $\kappa$-strategically closed, then it is $\kappa$-distributive.

\begin{lemma}
	\label{pstarpreservescardinals}
	Adopt the above notation. Then	$\bbP_*$ is $\lambda$-distributive and $\lambda^+$-c.c. In particular, forcing with $\bbP_*$ does not add any new sequences of ordinals
	of size less than $\lambda$ and it preserves all cardinals.
\end{lemma}
\begin{proof}
	Let us first show that $\bbP_*$ is $\lambda$-distributive. Thus suppose that $\mu < \lambda$
	and let $f: \mu \to V$ be a function in the generic extension by $\bbP_*.$ We will show that $f$ is in $V$.
	Let $\dot f$ be a $\bbP_*$-name for $f$ and let $p \in \bbP_*$ be such that
	\[
	p \Vdash \text{``}\dot f: \mu \to \check{V} \text{''}.
	\]
	By induction on $\xi \leq \mu$ we define a sequence $\langle  C_\xi: \xi \leq  \mu        \rangle$ of closed subsets of $\lambda$ and an increasing sequence  $\langle p_\xi: \xi \leq \mu         \rangle$
	of conditions in $\bbP_*$ such that:
	\begin{enumerate}
		\item $C_\xi$ is a closed subset of $\alpha_{p_\xi}+1$ with $\alpha_{p_\xi} \in C$, $C_\xi \cap \alpha_{p_\xi}$ is unbounded in
		$\alpha_{p_\xi}$ and $C_\xi \cap E_{p_{\xi}}=\emptyset,$
		
		\item $C_{\xi+1} \cap \alpha_{p_{\xi}}+1 = C_\xi,$
		
		\item if $\xi \leq \mu$ is a limit ordinal, then $C_\xi \cap \alpha_{p_\xi}=\bigcup\limits_{\zeta < \xi}C_{p_\zeta}$,
		
		\item $p_0=p,$
		
		\item for each $\xi < \mu, p_{\xi+1}$ decides $\dot{f}(\xi)$, say
		$p_{\xi+1} \Vdash$``$\dot f(\xi)=a_\xi$''.
		
		\item If $\xi \leq \mu$ is a limit ordinal and $\langle p_\zeta: \zeta < \xi        \rangle$
		is defined, then $p_\xi$ is the least upper bound of the sequence  $\langle p_\zeta: \zeta < \xi        \rangle$
		defined as follows:
		\begin{enumerate}
			\item $\alpha_{p_\xi}=\sup\limits_{\zeta < \xi}\alpha_\zeta$,
			
			\item $E_{p_\xi}= \big(\bigcup\limits_{\zeta < \xi}E_{p_\zeta} \cup \{ \alpha_{p_\xi} \} \big) \cap S^\lambda_{\aleph_0}$,
			
			\item for $s \in S$ and $\beta < \alpha_{p_\xi}, G^{p_\xi}_{s, \beta}=G^{p_\zeta}_{s, \beta}$ for some, and hence all, $\zeta < \xi$ with $\beta < \alpha_{p_\zeta}$,
			
			\item if $(s, t)\notin R$ and  $\beta < \alpha_{p_\xi}$, then  $\bold x^{p_\xi}_{s, t, \beta}=\bold x^{p_\zeta}_{s, t, \beta}$ for some, and hence all, $\zeta < \xi$ with $\beta < \alpha_{p_\zeta}$,
			
			\item for $s \in S$, $G^{p_\xi}_{s, \alpha_{p_\xi}}=\bigcup\limits_{\zeta < \xi}G^{p_\zeta}_{s, \alpha_{p_\zeta}}$,
			
			\item if $(s, t)\notin R$, then $\bold x^{p_\xi}_{s, t, \alpha_{p_\xi}}=\bigcup\limits_{\zeta < \xi}\bold x^{p_\zeta}_{s, t, \alpha_{p_\zeta}}$, which is defined in the natural way.
		\end{enumerate}
	\end{enumerate}
	Given $\xi \leq \mu$, let us show that $p_\xi$ as defined above is a condition, as then it is clear that $p_\xi$ extends all $p_\zeta$'s , for $\zeta < \xi.$
	Items (1) and (5) from Definition
	\ref{pstarforcing} are clearly satisfied.

	In order to see clause (2), it suffices to show that for each $s \in S,$  $G^{p_\xi}_{s, \alpha_{p_\xi}}$ is free.
	We have $C_{\xi} \cap \alpha_{p_\xi}$ is a club of $\alpha_{p_\xi}$ and
	\[
	G^{p_\xi}_{s, \alpha_{p_\xi}}=\bigcup\limits_{\gamma \in C_\xi \cap \alpha_{p_\xi}} G^{p_\xi}_{s, \gamma}.
	\]
	Furthermore, the sequence $\langle G^{p_\xi}_{s, \gamma}: \gamma \in C_\xi           \rangle$ is an increasing and continuous sequence of free groups, such that if
	$\gamma < \beta$ are successive points in $C_\xi,$ then $G^{p_\xi}_{s, \beta}/G^{p_\xi}_{s, \gamma}$ is free. It follows that $G^{p_\xi}_{s, \alpha_{p_\xi}}$
	is free as required.
	
	The set $E_{p_\xi}$ does not reflect as witnessed by
	$C_\xi$, hence clause (3) is satisfied.
	
	To show that clause (4) holds, let $\gamma <\beta < \alpha_{p_\xi}$
	with $\gamma \notin E_{p_\xi}.$ Let $\zeta < \xi$ be such that $
	\beta < \alpha_{p_\zeta}$. This implies $\gamma \notin E_{p_\zeta}$, and
	consequently
	\[
	G^{p_\xi}_{s, \beta} / G^{p_\xi}_{s, \gamma} = G^{p_\zeta}_{s, \beta} / G^{p_\zeta}_{s, \gamma}
	\]
	is free.
	
	Let $q=p_\mu,$ and let $h=\langle  a_\xi: \xi < \mu       \rangle$. Then $h \in V$
	and $q \Vdash$``$\dot {f}=h$''. As $p$ was arbitrary, we are done.

	Now as $\lambda^{<\lambda}=\lambda$ and $\bbP_* \subseteq \cH(\lambda)$, we have $|\bbP_*|=\lambda$ and hence it clearly satisfies
	the $\lambda^+$-c.c.
\end{proof}
\begin{remark}
	The above proof shows that the forcing notion $\bbP_*$ is indeed $\lambda$-strategically closed.
\end{remark}
\begin{lemma}
	\label{Estationary}Suppose $\mathbf G_*\subseteq \bbP_*$ is generic over $V$.
	Then $E:=\bigcup\limits_{p \in \mathbf G_*}E_p$ is a non-reflecting stationary subset of $\lambda.$
\end{lemma}
\begin{proof}
	This is standard, so we just sketch the proof. Let $p  \in \bbP_*$ and suppose that $$p\Vdash``\dot{C} \subseteq \lambda\emph{
	is a club of } \lambda.''$$  Let $\chi > \lambda^+$ be large enough regular, $\lhd$ be a well-ordering of $\cH(\chi)$ and let $M \prec (\cH(\chi), \in, \lhd)$ be an elementary submodel of
	$\cH(\chi)$ such that:
	\begin{enumerate}
		\item $|M| < \lambda$,
		\item $p, \dot{C}, \bbP_*, \cdots \in M,$
		\item $M \cap \lambda=\delta$ for some $\delta \in S^\lambda_{\aleph_0}$.
	\end{enumerate}
	
	By induction on $n<\omega$ we define an increasing sequence $\langle p_n: n<\omega    \rangle$
	of conditions, together with a sequence $\langle C_n: n<\omega    \rangle$ satisfying the following:
	\begin{itemize}
		\item[(4)] $p \leq p_0,$
		\item[(5)] $p_{n+1}$ decides $\dot{C} \cap \alpha_{p_n}$ to be $C_n$,
		\item[(6)] $p_{n+2} \Vdash$``$\dot{C} \cap (\alpha_{p_n}, \alpha_{p_{n+1}}) \neq \emptyset$'',
		\item[(7)] $\sup\limits_{n<\omega}\alpha_{p_n}=\delta.$
	\end{itemize}
	We are now ready to define $$q:=\langle  \langle\mathbb{G}^q_{s, \beta}: s \in S, \beta \leq \delta \rangle,  E_q, \langle \mathbf x^q_{s,t,\beta}: (s, t) \notin R, \beta \leq \delta             \rangle\rangle,$$as follows:
	\begin{enumerate}
		\item[(8)] for $s \in S$ and $\beta < \delta, \mathbb{G}^q_{s, \beta}=\mathbb{G}^{p_n}_{s, \beta},$
		where $n$ is such that $\alpha_{p_n}>\beta$,
		\item[(9)] for $s\in S, \mathbb{G}^q_\delta=\bigcup\limits_{n<\omega}\mathbb{G}^{p_n}_{s, \alpha_{p_n}}$,
		
		\item[(10)] $E_q = \bigcup\limits_{n<\omega}E_{p_n} \cup \{\delta\}$,
		\item[(11)] if $(s, t) \notin R$ and $\beta<\delta,$ $\mathbf x^q_{s,t,\beta}=\mathbf x^{p_n}_{s,t,\beta}$, for some $n$ with $\alpha_{p_n} > \beta$,
		\item[(12)] $\mathbf x^q_{s,t,\delta}$ is the exact sequence that is the direct limit of the sequence $\mathbf x^{p_n}_{s,t,\alpha_{p_n}}$.
	\end{enumerate}

	It is easily seen that $q \in \bbP_*$ is well-defined and it forces $\delta \in \dot{C} \cap \dot{E}$,
	which completes the proof.
\end{proof}
Recall the following easy fact:

\begin{fact}\label{subfree}
	(See \cite{fuchs})	Any subgroup of a free abelian group is free.\end{fact}

For each $s \in S$ and $\beta < \lambda$, we set $\mathbb{G}_{s, \beta}=\mathbb{G}^p_{s, \beta}$ for some (and hence any) $p \in \mathbf G_*$
with $\alpha_p \geq \beta$. Also, we set
\[
\mathbb{G}_s:=\bigcup\limits_{\beta < \lambda}\mathbb{G}_{s, \beta}.
\]
Let us first show that $\mathbb G_s$ is not free.
\begin{lemma}
	\label{Gsproperties}
	Suppose $s \in S$. Then $\mathbb{G}_s$ is a non-free strongly $\lambda$-free abelian group of size $\lambda.$
\end{lemma}
\begin{proof}
	It is clear from Definition \ref{pstarforcing}(a)(2) that $\mathbb{G}_s$ is a strongly $\lambda$-free abelian group of size $\lambda.$
	Let us show that it is not free.
Recall that we are	in $\bold V,$ and let $\mathbb{A}$ be a $\lambda$-free abelian group of size $\lambda$ which is not free. We may assume that the universe of $\mathbb{A}$ is $\lambda.$ Let also $ \langle  A_\alpha: \alpha < \lambda       \rangle$
	be a filtration of $\mathbb A$. It follows that the set
	\[
	E_{\mathbb A}:=\{ \alpha < \lambda: \mathbb A/ \mathbb A_\alpha \text{~in not~}\lambda\text{-free}           \}
	\]
	is stationary. We apply an argument similar to the proof of Lemma
	\ref{Estationary}, and we are able to show  $E_{\mathbb A}$ remains stationary in $\bold V[\bold{G}_*]$. In  particular,  $\mathbb A$ remains
	non-free in the generic extension $\bold V[\bold{G}_*]$.

	For each $\gamma < \lambda$ let $D_\gamma$ be the set
	of all
	$p \in \bbP_*$ such that there exists $\beta$ such that:
	\begin{itemize}
		\item  $\alpha_p=\beta+1 >\gamma$,
		\item $\mathbb{G}^p_{s, \beta+1}$ includes  $\mathbb{A}_\beta$ as a subgroup.
	\end{itemize}
	It is easily seen that each set $D_\gamma$ is dense in $\bbP_*.$ Now for each $\gamma< \lambda$  pick some $p_\gamma \in D_\gamma \cap \bold G_*$
	and set $\alpha_{p_\gamma}=\beta_\gamma+1 > \gamma$.

	Now we look at the following commutative diagram:
	
	\[
	\begin{CD}
	\ldots @>\subseteq>> \mathbb{A}_{\beta_\gamma}@>\subseteq>>\mathbb{A}_{\beta_{\gamma+1}}	  @>\subseteq>>  \ldots\\
	@V \subseteq VV  @V \subseteq VV   @VV\subseteq V  \\
	\ldots @>\subseteq>> \mathbb{G}^{p_\gamma}_{s, \beta_{\gamma+1}}  @>\subseteq >> \mathbb{G}^{p_{\gamma+1}}_{s, \beta_{\gamma+1}+1} @>\subseteq>> \ldots
	\end{CD}
	\]
	Taking direct limits of these directed systems, lead us  to a natural inclusion map
	\[
	\mathbb{G}_{s}=\bigcup\limits_{\gamma < \lambda} \mathbb{G}^{p_\gamma}_{s, \beta_\gamma+1} \supseteq \bigcup\limits_{\gamma < \lambda} \mathbb{A}_{\beta_\gamma}= \mathbb{A}.
	\]
	As $\mathbb{A}$ is non-free, it follows from Fact \ref{subfree} that
	$\mathbb{G}_s$ is also non-free. This completes the proof.
\end{proof}

For each $(s, t) \in (S\times S)\setminus R$, we look at the following commutative diagram of short exact sequences:
$$
\begin{CD}
@.\vdots@.\vdots@.\vdots\\
@.\subseteq@AAA\subseteq @AAA \subseteq @AAA   \\
\mathbf x_{s, t, \beta}:= 0@>>> \mathbb{G}_{t,\beta} @>f_{s,t, \beta}>>\mathbb{H}_{s,t, \beta} @>g_{s, t, \beta}>> \mathbb{G}_{s,\beta} @>>> 0\\
@. \subseteq @AAA \subseteq@AAA \subseteq @AAA   \\
@.\vdots@.\vdots@.\vdots\\
@.\subseteq @AAA\subseteq@AAA \subseteq @AAA   \\
\mathbf x_{s, t, 1}:=0@>>> \mathbb{G}_{t,1} @>f_{s,t, 1}>>\mathbb{H}_{s,t, 1} @>g_{s, t, 1}>> \mathbb{G}_{s,1} @>>> 0\\
@.\subseteq@AAA\subseteq @AAA \subseteq @AAA   \\
\mathbf x_{s, t, 0}:=0@>>> \mathbb{G}_{t,0} @>f_{s,t,0 }>>\mathbb{H}_{s,t, 0} @>g_{s, t, 0}>> \mathbb{G}_{s,0} @>>> 0,\\
\end{CD}
$$
where for each $\beta<\lambda, \mathbf x_{s, t, \beta}=\mathbf x_{s, t, \beta}^p$, for some and hence any $p \in \bold G_*$ with $\alpha_p \geq \beta.$ By taking the corresponding inductive limit, we lead  to the following short exact sequence
$$
\begin{CD}
\mathbf x_{s,t} :=0@>>> \mathbb{G}_{t} @>f_{s,t}>>\mathbb{H}_{s,t} @>g_{s, t}>> \mathbb{G}_{s} @>>> 0.
\end{CD}
$$ In other words,
$\mathbf x_{s,t}:=\lim_{\beta \to \lambda} \mathbf x^p_{s, t, \beta}$.
The next lemma shows that $\mathbf x_{s,t}$ does not split.
\begin{lemma}
	\label{nosplitifincompatible}
	Suppose $s_*, t_* \in S$ and $(s_*, t_*) \notin R$. Then in $V[\mathbf G_*]$ the exact sequence
	$\mathbf x_{s_*,t_*}$ does not split. In particular, $\Ext(\mathbb{G}_{s_*}, \mathbb{G}_{t_*}) \neq 0.$
\end{lemma}
\begin{proof}
	Suppose towards contradiction that the exact sequence
	$\mathbf x_{s_*,t_*}$ splits and let $p \in \bbP_*$ and $\name{h}$ be such that
	\begin{center}
		$p \Vdash$``$\name{h}: \name{\mathbb{G}}_{s_*} \to \name{\mathbb{H}}_{s_*,t_*}$ is such that $\name h \circ \name{g}_{s_*,t_*}=\id_{\name{\mathbb{G}}_{s_*}}$''.
	\end{center}
	As in the proof of Lemma \ref{Estationary}, we can find an extension $q \leq p$ such that
	\begin{enumerate}
		\item $\alpha_q\in E_q \cap S^\lambda_{\aleph_0},$
		\item $q$ decides $\name{h} \restriction \mathbb G^q_{s_*, \alpha_q},$ say
		\[
		q \Vdash \text{``} \name{h} \restriction \mathbb G^q_{s_*, \alpha_q} = h_* \text{''}.
		\]
	\end{enumerate}
	We will need the following claim.

	\begin{claim}
\label{splitcontra}
Let		$$
\begin{CD}
\mathbf x_{s_*, t_*, \alpha_q}:=0@>>> \mathbb{G}^q_{t_*,\alpha_q} @>f^q_{s_*,t_*,\alpha_q }>>\mathbb{H}^q_{s_*,t_*, \alpha_q} @>g_{s_*, t_*, \alpha_q}>> \mathbb{G}^q_{s_*,\alpha_q} @>>> 0\\
\end{CD} $$ be given and suppose
$h_*$ is an splitting of
$g_{s_*, t_*, \alpha_q}$.
Then there exists an exact sequence $\mathbf x_{s_*, t_*}$ that fits in the following commutative diagram
		$$
		\begin{CD}
		\mathbf x_{s_*, t_*}:=0@>>> \mathbb{G}'_{t_*} @>f'>>\mathbb{H}'_{s_*,t_*} @>g'>> \mathbb{G}'_{s_*} @>>> 0\\
		@.@AAA @AAA @AAA   \\
		\mathbf x_{s_*, t_*, \alpha_q}:=0@>>> \mathbb{G}^q_{t_*,\alpha_q} @>f^q_{s_*,t_*,\alpha_q }>>\mathbb{H}^q_{s_*,t_*, \alpha_q} @>g_{s_*, t_*, \alpha_q}>> \mathbb{G}^q_{s_*,\alpha_q} @>>> 0\\
		\end{CD}
		$$
		such that there exists no splitting $h'$ of $g'$ extending $h_*$. Furthermore, there is an extension $r$ of $q$
such that $\alpha_ r=\alpha_q+1$ and $\bold x^r_{s_*, t_*, \alpha_r}=\bold x_{s_*, t_*}$.
	\end{claim}
	\begin{proof}
Without loss of generality, either $\alpha_p$ is a cardinal, or $|\alpha_p|<\alpha_p $ and $\alpha_p^\omega=\alpha_p$. 	Since $\cf(\alpha_p)=\omega$, there is an increasing sequence $\langle \gamma_m: m<n \rangle$ which is cofinal in $\alpha_p$ with $\gamma_m\notin E_q$. To simplicity, let $(s_1,s_2)=(t_*,s_*)$
		and for $\ell=1,2$ we set $\mathbb{G}_{s_{\ell}}:=\mathbb{G}^q_{s_{\ell},\alpha_q}$.
	We have $$\mathbb{G}_{s_\ell}=\bigcup\limits_{n<\omega} \mathbb{G}^q_{s_{\ell},\gamma_n},$$ where for each $n$, both $\mathbb{G}^q_{s_{\ell},\gamma_n}$
and $\mathbb{G}^q_{s_{\ell},\gamma_{n+1}}/ \mathbb{G}^q_{s_{\ell},\gamma_n}$ are free, so
 we can find a free basis $\vec{x}_{\ell}:=\langle x_\alpha^{\ell}: \alpha<\alpha_q \rangle$
	of $\mathbb{G}_{s_{\ell}}$ such that  $\vec{x}_{\ell}\rest \gamma_n$ is a
	free basis of 	$\mathbb{G}^q_{s_{\ell},\gamma_n}$.
	Without loss of generality, we may and do assume that $f^q_{s,t,\alpha_q }$ is the natural inclusion morphism:
	$$\mathbb{G}_{{s_{1}}} \stackrel{f^q_{s,t,\alpha_q }}\subseteq\mathbb{H}^q_{{s_{2},s_{1}}, \alpha_q},$$
		and $h_\ast(x^2_\alpha)=x^2_\alpha$ for any $x^2_\alpha
\in \mathbb{G}_{s_{2}}$.

We define the abelian group
		$\mathbb{G}_{s_{\ell}}'$  to be generated by $$X_{\ell}:=\mathbb{G}_{s_{\ell}}	\cup\{z^\ell_n:n<\omega\}$$freely except the equations below which 	 $\mathbb{G}_{s_{\ell}}'$
		satisfy:
		\begin{itemize}
		\item[$(\ast)_1$]: $n!z^\ell_{n-1}=z^\ell_n+ x^\ell_{\gamma_{n}}$.
	\end{itemize}
		Recall that $\langle x^1_\alpha,x^2_\alpha: \alpha<\alpha_q \rangle$
		is a free basis of $\mathbb{H}_{s_{1},s_{2}}$.
			 We let
			 $\mathbb{H}_{s_{1},s_{2}}'$be the abelian group generated by
		$$\mathbb{H}_{s_{1},s_{2}}\cup\{z_n:n<\omega\},$$ freely except the equations below
			 which 	$\mathbb{H}_{s_{1},s_{2}}$ should satisfy:
		\begin{itemize}
		\item[$\odot_1$]: $n!z_{n-1}=z_n+ x^1_{\gamma_{n}}+ x^2_{\gamma_{n}}$.
	\end{itemize}	
	
	First, we show that these new abelian groups $\{\mathbb{G}_{s_{1}}',\mathbb{H}_{s_{1},s_{2}}',\mathbb{G}_{s_{2}}'\}$ are free, by presenting their bases:
		\begin{itemize}\item[$(b)_1$]: The group $\mathbb{G}_{s_{1}}'$ is free. Indeed, in view of $(\ast)_1$ we observe that  $z^1_n=n!z^1_{n-1}-x^1_{\gamma_{n}}.$   An easy inductive argument on $\ell$ yields that $z^1_n\in \langle x^\ell_{\gamma_{n}},z^1_{0}\rangle,$  and so  the set $$B_1:=\{x^1_\alpha: \alpha < \alpha_q\} \cup \{z^1_0\}$$ generates $\mathbb{G}_{s_{1}}'$. Since there is no relation involved in $B_1$ we deduce that it is a free base of $\mathbb{G}_{s_{1}}'$. Thus, $\mathbb{G}_{s_{1}}'$ is free, as claimed.
		 \item[$(b)_2$]: The group $\mathbb{G}_{s_{2}}'$ is free. Indeed, in  $(b)_1$ replace $B_1$ with $$B_2:=\{x^1_\alpha: \alpha < \alpha_q\} \cup \{z^2_0\}$$ and conclude that $\mathbb{G}_{s_{2}}'$ is free.
		 \item[$(b)_3$]: The group  $\mathbb{H}_{s_{1},s_{2}}'$ is free. Indeed, in view of $\odot_1$ the set $$B_3:=\{x^1_\alpha, x^2_\alpha: \alpha < \alpha_q\} \cup \{z_0\}$$ generates $\mathbb{H}_{s_{1},s_{2}}'$. Since there is no relation involved in elements of $B_3$ we deduce that it is a free base of $\mathbb{H}_{s_{1},s_{2}}'$. Thus, $\mathbb{H}_{s_{1},s_{2}}'$ is free, as claimed.
	\end{itemize}

Now, we are going to define a map $f'$ (resp.  $g'$)
	extending $\id_{\mathbb{G}_{s_{1}}}$ (resp. $g=g_{s_1,s_2, \alpha_q}) $ such that the following data becomes an exact sequence	$$
	\begin{CD}
0@>>> \mathbb{G}_{s_{1}}' @>f'>>\mathbb{H}_{s_{1},s_{2}}' @>g'>> \mathbb{G}_{s_{2}}' @>>> 0.\\
	\end{CD} $$Since we need $f'$  extends $f$, we  define $f'(x^1_{\alpha}):=x^1_{\alpha}$.
 In order to complete the definition of $f'$, we proceed by induction on $n$ to define $f'(z^1_n)$. When $n = 0$, we set $$f'(z^1_0):=z_1-z_0-x^2_{\gamma_0}\quad(\dagger).$$ Now suppose  $f'(z^1_{n-1})$ is defined.
We use the equation $n!z^1_{n-1}=z^\ell_n+x^1_{\gamma_{n}}$ and define
$$f'(z^1_{n}):=n!f'(z^1_{n-1})-x^1_{\gamma_{n}}\quad(+)$$
Also, the assignments $z_n\mapsto  z^2_n$, $x^1_\alpha \mapsto  0$ and $x^2_\alpha \mapsto  x^2_\alpha$  define a map $g':\mathbb{H}_{s_{1},s_{2}}'\to\mathbb{G}_{s_{2}}' $ which extends $g$. This is well-defined. Namely,	 it sends the relation $$n!z_{n-1}=z_n+ x^1_{\gamma_{n}}+ x^2_{\gamma_{n}}$$ from $\odot_1$ into the relation
 $n!z^2_{n-1}=z^2_n+ x^2_{\gamma_{n}}$ from  $(\ast)_1$.

Let us  show that $\Rang(f')\subseteq \ker(g')$. To see this, first note that $g'(f'(z^1_0))=0$. Now suppose by induction that  $g'(f'(z^1_{n-1}))=0$. Then $$g'(f'(z^1_{n}))\stackrel{(+)}=n!g'(f'(z^1_{n-1}))n!-g'(x^1_{\gamma_n})=-g'(x^1_{\gamma_n})=0.$$
We proved that $$\Rang(f')\subseteq \ker(g')\quad(\star)$$ To see the reverse inclusion, we revisit $\odot_1$
and note that $\ker(g')$ is generated by $$\mathbb{G}_{s_1} \cup \{k_n:=n!z_{n-1}-z_n- x^2_{\gamma_{n}}\}.$$
As $g'$ extends $g$, we have $$\mathbb{G}_{s_1} = \Rang(f)\subseteq \Rang(f')\stackrel{(\star)}\subseteq \ker(g').$$ By  induction on $n$, we show that $k_n\in\Rang(f')$. Following $(+)$ it is enough to deal with $n=0$,
and in this case $k_0=f'(z_0^1)$. In sum, we proved that
$$
\begin{CD}
0@>>> \mathbb{G}_{s_{1}}' @>f'>>\mathbb{H}_{s_{1},s_{2}}' @>g'>> \mathbb{G}_{s_{2}}' @>>> 0\\
\end{CD} $$is an exact sequence of free abelian groups.
	
	Assume toward the contradiction that
	there exists $h'\in\Hom(\mathbb{G}_{s_{2}}',\mathbb{H}_{s_{1},s_{2}}')	$
		such that $h'\supseteq h_\ast$.
		Let $z^\ast_n:=h'(z^2_n)\in \mathbb{H}_{s_{1},s_{2}}'$. As	
			$$\mathbb{G}_{s_{2}}'\models n!z^2_{n-1}=z_n^2+x^2_{\gamma_{n}},$$
			we have
				\begin{itemize}
				\item[$\odot_2$]: $\mathbb{H}_{s_1,s_{2}}'\models n!z^\ast_{n-1}=z_n^\ast+ h_\ast(x^2_{\gamma_{n}}).$
			\end{itemize}	
			Subtract $\odot_1$-$\odot_2$ we get
			\begin{itemize}
			\item[$\odot_3$]: $$\mathbb{H}_{s_1,s_{2}}'\models n!(z_{n-1}-z^*_{n-1})=(z_{n}-z^*_n)+  x^1_{\gamma_{n}}.$$
		\end{itemize}
	Noting that $g'(z_{n}-z^\ast_{n})=z^2_n-z^2_n=0$. From this, $z_n - z^*_n=f'(y_n)$ for some $y_n\in\mathbb{G}_{s_{1}}'$. Recall that $f'$ is injective. Apply this along with $\odot_3$ and deduce that $$\mathbb{G}_{s_{1}}'\models n!y _{n-1}=y_{n}+  x^1_{\gamma_{n}}.$$
	Due to the uniqueness of solutions of $(\ast)_1$ we conclude that $$f'(z^1_n)=f'(y_n)=z_n- z_n^\ast=z_n-h'(z^2_n).$$In other words, we determined $h'$, that is:
	$$h'(z^2_n)=z_n-f'(z^1_n).$$
We substitute this  from $(\dagger)$ and evaluate $g'$
	on the both sides of that equation,
	then we  observe that
	\begin{equation*}
	\begin{array}{clcr}
	z^2_0&=g'(h'(z^2_0))\\ &= g'(z_0-f'(z^1_0))\\
	&
	= g'(z_0-(z_1-z_0-x^2_{\gamma_0}))\\
	&
	=z_0^2-z_1^2+z_0^2+x^2_{\gamma_0}.
	\end{array}
	\end{equation*}
	Thus
\[
z_0^2= z_1^2-x^2_{\gamma_0}.
\]
But recalling from $(\ast)_1$,
\[
z_0^2=z_1^2+x^2_{\gamma_0},
\]
which imply $x^2_{\gamma_0}=0,$ a contradiction.
This contradiction shows that a such  $h'$ does not exist.

Finally let  $r$ be the extension of $q$ such that:
\begin{itemize}
	\item[$r_1)$]: $\alpha_r=\alpha_q+1$,
	
	\item[$r_2)$]: $\mathbb{G}^r_{s, \alpha_r}=\mathbb{G}^q_{\alpha_q}$ if $s \notin \{s_*, t_*\}$,
	
	\item[$r_3)$]: $\mathbf x^r _{s_*, t_*, \alpha_r}=\mathbf x_{s_*, t_*},$ where $\mathbf x_{s_*, t_*}$ is the exact sequence defined above,
	
	\item[$r_4)$]: $\mathbf x^r _{s, t, \alpha_r}=\mathbf x^q_{s, t, \alpha_q}$ for all $(s, t) \notin R$ such that $ s \neq s_*$ and  $t \neq t_*,$
	
	\item[$r_5)$]: if $s=s_*$ and $t \neq t_*,$ then the desired sequence
	$\mathbf x^r_{s_*, t, \alpha_r}$ is defined as follows $$
	\begin{CD}
	0@>>> \mathbb{G}^r_{t, \alpha_r} @>f^r_{s_*, t, \alpha_r}>>\mathbb{H}^r_{s_*, t, \alpha_r} @>g^r_{s_*, t, \alpha_r}>> \mathbb{G}^r_{s_*, \alpha_r} @>>> 0,\\
	\end{CD} $$
	where
	\begin{itemize}
		\item[$(r_5.i)$]: $\mathbb{H}^r_{s_*, t, \alpha_r}=\frac{\mathbb{H}^q_{s_*, t, \alpha_q} \oplus \bigoplus_{n<\omega}z_n \mathbb{Z}}{\langle n!z_{n-1}-z_n-x^2_{\gamma_n}    \rangle}$,
		\item[$(r_5.ii)$]: $f^r_{s_*, t, \alpha_r}=f^q_{s_*, t, \alpha_q}$,
		\item[($r_5.iii)$]: $g^r_{s_*, t, \alpha_r} \restriction \mathbb{H}^q_{s_*, t, \alpha_q}=g^q_{s_*, t, \alpha_q}$ and $g^r_{s_*, t, \alpha_r}(z_n)=z^2_n,$
	\end{itemize}
	\item[$r_6)$]: In the case $s \neq s_*$ and $t = t_*,$  
the proposed sequence	$\mathbf x^r_{s, t_*, \alpha_r}$ is defined as follows $$
	\begin{CD}
	0@>>> \mathbb{G}^r_{t_*, \alpha_r} @>f^r_{s, t_*, \alpha_r}>>\mathbb{H}^r_{s, t_*, \alpha_r} @>g^r_{s, t_*, \alpha_r}>> \mathbb{G}^r_{s, \alpha_r} @>>> 0,\\
	\end{CD} $$
	where 
	\begin{itemize}
		\item[$(r_6.i)$]: $\mathbb{H}^r_{s, t_*, \alpha_r}=\frac{\mathbb{H}^q_{s, t_*, \alpha_q} \oplus \bigoplus_{n<\omega}z_n \mathbb{Z}}{\langle n!z_{n-1}-z_n-x^1_{\gamma_n}    \rangle}$,
		\item[($r_6.ii)$]:  $f^r_{s, t_*, \alpha_r} \restriction \mathbb{G}^q_{t_*, \alpha_q}=f^q_{s_*, t, \alpha_q}$
		and $f^r_{s, t_*, \alpha_r}(z^1_n)=0$, 
		\item[$(r_6.iii)$]: $g^r_{s_*, t, \alpha_r} \restriction \mathbb{H}^q_{s_*, t, \alpha_q}=g^q_{s_*, t, \alpha_q}$ and $g^r_{s_*, t, \alpha_r}(z_n)=0.$
	\end{itemize}
\end{itemize}
It is easily seen that $r$ as defined above is a condition extending $q$. 
\end{proof}
Let us continue the proof of Lemma \ref{nosplitifincompatible}.	
By the way we defined $r$, see the above itemized properties $r_1),\ldots,r_6)$, we have
\begin{center}
$r \Vdash$``there is no splitting map for $\name{g}_{s_*,t_*}$ extending $h_*$''.
\end{center}
In order to see this, let $r' \leq r$ be such that $r'$ decides $\name{h} \restriction \mathbb{G}^{r'}_{s_*, \alpha_r}$. Then
\[
r'\Vdash \text{``} \name{h} \restriction \mathbb{G}^{r'}_{s_*, \alpha_r}= g'  \text{''}.
\]

This contradicts the choice of $g'$. The argument of Lemma \ref{nosplitifincompatible} is now completed.	
\end{proof}

Let us now work in the generic extension
$V[\mathbf G_*]$. We  define a $\lambda$-support iteration
\[
\bbP=\langle  \langle \bbP_\alpha: \alpha \leq \lambda^+ \rangle, \langle \dot{\bbQ}_\beta: \beta < \lambda^+ \rangle                         \rangle
\]
of forcing notions which forces $\Ext(\mathbb{G}_s, \mathbb{G}_t) = 0$ for all  $s, t \in S$ with $s R t$. We first define the building blocks of this iteration.

Suppose $W \supseteq V[\mathbf{G}_*]$ is a forcing extension of $V[\mathbf{G}_*]$, $s, t \in S$
are such that $s R t$ and suppose that
$$
\begin{CD}
\mathbf x :=0@>>> \mathbb{G}_{s} @>f>>\mathbb{H} @>g>> \mathbb{G}_{t} @>>> 0
\end{CD}
$$
is an exact sequence in $W$.
\begin{definition}
\label{buildingblockforcing}
The forcing notion $\bbQ^W_{s, t, \mathbf{x}}$ consists of partial functions $q:\mathbb{G}_{t} \dashrightarrow \mathbb{H}$ with domain $\dom(q)$
such that:
\begin{enumerate}
\item $\dom(q)= \mathbb{G}_{s, \gamma}$, where $\gamma \notin E$,

\item $g \circ q = \id_{\dom(q)}$.
\end{enumerate}
$\bbQ^W_{s, t, \mathbf{x}}$ is ordered by inclusion.
\end{definition}

Thus, the forcing notion $\bbQ^W_{s, t, \mathbf{x}}$ aims to add a splitter for the exact sequence $\mathbf x$.
The next lemma shows that this is indeed possible.
\begin{lemma}
\label{qwforcingproerties}
The forcing notion $\bbQ^W_{s, t, \mathbf{x}}$ is $\lambda$-closed, $\lambda^+$-c.c., and forcing with it adds a function
$h: \mathbb{G}_t \to \mathbb{H}$ such that $g \circ h=\id_{\mathbb{G}_t}$.
\end{lemma}
\begin{proof}
The fact that $\bbQ^W_{s, t, \mathbf{x}}$  is $\lambda$-closed follows from a combination of Definition  \ref{buildingblockforcing}(1)
along with Definition \ref{pstarforcing}(4). This allows us to take unions (or direct limits) at limit stages and still have a condition. According to $\Delta$-system lemma,
the forcing is $\lambda^+$-c.c., and claimed.
\end{proof}

We are finally ready to define our iteration.
Let $\beta< \lambda^+$ and suppose that $\bbP_\beta$ is defined. If
$\Phi(\beta)$ is a $\bbP_* \ast \dot{\bbP}_\beta$-name for a triple $(s, t, \dot{\mathbf x}),$ where
$s, t \in S, s R t$ and $\dot{\mathbf x}$ is a name for an exact sequence
$$
\begin{CD}
\mathbf x :=0@>>> \mathbb{G}_{s} @>\dot{f}>>\dot{\mathbb{H}} @>\dot{g}>> \mathbb{G}_{t} @>>> 0,
\end{CD}
$$
in  $\cH(\lambda^+)$, then
\[
\Vdash_{\bbP_\beta}\text{``} \dot{\bbQ}_\beta=\dot{\bbQ}^{V^{\bbP_* \ast \dot{\bbP}_\beta}}_{s, t, \dot{\mathbf{x}}}\text{''.}
\]
Otherwise, let $\dot{\bbQ}_\beta$ be forced to be the trivial forcing notion.

The next lemma follows from \cite{Sh:587}.
\begin{lemma}
\label{s12}
Work in $V[\mathbf G_*].$ The forcing notion $\bbP=\bbP_{\lambda^+}$ is $\lambda$-complete and $\lambda^+$-c.c.
\end{lemma}
It follows from the above lemma that forcing with $\bbP$ preserves all cardinals and adds no new sequences of ordinals of length less than $\lambda$.
Suppose
\[
\mathbf G = \langle \langle \mathbf G_\alpha: \alpha \leq \lambda^+      \rangle, \langle \mathbf H_\beta: \beta < \lambda^+   \rangle\rangle
\]
is $\bbP$-generic over $V[\mathbf G_*]$.
\begin{lemma}
\label{whensRt}
Work in $V[\mathbf G_* \ast \mathbf G]$. If $s, t \in S$ and $s R t$,  then $\Ext(\mathbb{G}_s, \mathbb{G}_t)= 0$.
\end{lemma}
\begin{proof}
It suffices to show that any exact sequence
$$
\begin{CD}
\mathbf x :=0@>>> \mathbb{G}_{s} @>f>>\mathbb{H} @>g>> \mathbb{G}_{t} @>>> 0
\end{CD}
$$
with $\mathbb{H}$ of size $\lambda$ splits. Let $\dot{\mathbf x}$ be a $\bbP_* \ast \bbP$-name for $\mathbf x$ which we may assume that $\dot{\mathbf x} \in \cH(\lambda^+)$. Furthermore, we can find some $\alpha<\lambda^+$ such that
$\dot{\mathbf x}$ is a $\bbP_* \ast \bbP_\alpha$-name, and then by the choice of $\Phi,$ we may also assume that $\Phi(\alpha)=\dot{\mathbf x}$.
By definition of  forcing notion, we know
$$\Vdash_{\bbP_*\ast\bbP_{\alpha+1}}\text{``} \dot{\mathbf x} \text{~splits~}\text{'',}$$
and consequently
\[
\Vdash_{\bbP_*\ast\bbP}\text{``} \dot{\mathbf x} \text{~splits~}\text{''.}
\]
This completes the proof.
\end{proof}
Let $(s, t) \in (S \times S)\setminus R$. We now show that the iteration does not add splitters for $\mathbf x_{s,t}$,
\begin{lemma}
\label{snotRtcase}
Suppose $(s, t) \in (S \times S)\setminus R$. Then the exact sequence
$$
\begin{CD}
\mathbf x_{s,t} :=0@>>> \mathbb{G}_{s} @>f_{s,t}>>\mathbb{H}_{s,t} @>g_{s, t}>> \mathbb{G}_{t} @>>> 0
\end{CD}
$$
does not split in  $V[\mathbf G_* \ast \mathbf G]$.
\end{lemma}
\begin{proof}
Since the argument  is similar to the proof of Lemma \ref{nosplitifincompatible}, we just present the sketch of proof.
Here, we work in $V$.
Let us first
combine Lemma  \ref{pstarpreservescardinals} along with Lemma
\ref{s12}, and
 note that the set of conditions of the form $$(p, \langle q_\alpha: \alpha < \lambda^+             \rangle) \in \bbP_* \ast \bbP_{\lambda^+}$$ such that $p, q_\alpha \in V$, is dense in $\bbP_* \ast \bbP_{\lambda^+}$.
Suppose towards contradiction that the exact sequence
$\mathbf x_{s,t}$ splits. This gives us $$(p, \langle p_\alpha: \alpha < \lambda^+             \rangle) \in \bbP_* \ast \bbP_{\lambda^+}$$ and $\name{h}$ satisfying:
\begin{center}
	$(p, \langle p_\alpha: \alpha < \lambda^+             \rangle) \Vdash$``$\name{h}: \name{\mathbb{G}}_s \to \name{\mathbb{H}}_{s,t}$ is such that $\name h \circ \name{g}_{s,t}=\id_{\name{\mathbb{G}}_s}$''.
\end{center}
As before, we can find an extension $$(p, \langle p_\alpha: \alpha < \lambda^+             \rangle) \leq (q, \langle q_\alpha: \alpha < \lambda^+             \rangle)$$ equipped with the following two properties:
\begin{enumerate}
	\item $\alpha_{q}\in E_{q} \cap S^\lambda_{\aleph_0},$
	\item $(q, \langle q_\alpha: \alpha < \lambda^+             \rangle)$ decides $\name{h} \restriction \mathbb G^{q}_{s, \alpha_{q}},$ say
	\[
	(q, \langle q_\alpha: \alpha < \lambda^+             \rangle) \Vdash \text{``} \name{h} \restriction \mathbb G^{q}_{s, \alpha_{q}} = h_* \text{''}.
	\]
\end{enumerate}
In view of Claim \ref{splitcontra}  we can find an extension  $r$   of $q$ such that
$(r, \langle q_\alpha: \alpha < \lambda^+             \rangle)$ extends $(q, \langle q_\alpha: \alpha < \lambda^+             \rangle)$
and also, it forces $$``\emph{there is no splitting map for }\name{g}_{s,t} \emph{ extending }h_*.''$$ We get a contradiction and the
lemma follows.
\end{proof}
This completes the proof of Theorem \ref{arbitrarygraphZFC}.

\end{document}